\documentclass[dvipsnames, 12pt]{scrartcl}
\usepackage{pgfplots}
\usepackage{subcaption}
\usepackage[english,ngerman]{babel} 
\usepackage{amsmath}
\usepackage{mathtools}
\usepackage{multirow}
\usepackage{tablefootnote}
\usepackage{bm}
\usepackage[capitalise]{cleveref}
\usepackage[standard]{ntheorem}
\usepackage{tikz}
\usepackage[draft, margin]{fixme}
\fxsetup{theme=color}
\usepackage{graphicx}
\usepackage[algoruled]{algorithm2e}
\usepackage{grffile}
\usepackage[backend=biber,style=numeric,maxbibnames=99]{biblatex}
\usepackage{enumitem}
\usepackage{float}
\usepackage{nicefrac}
\usepackage{csquotes}
\usepackage{booktabs}
\usepackage{textcomp}
\usepackage{accents}
\usepackage{bbm}
\usepackage{csquotes}

\FXRegisterAuthor{pg}{apg}{PG}
\FXRegisterAuthor{mf}{amf}{MF}
\FXRegisterAuthor{fk}{afk}{FK}
\FXRegisterAuthor{aa}{aaa}{AA}

\theoremstyle{remark}

\renewtheorem{proposition}{Proposition}
\crefname{assumption}{Assumption}{Assumptions}

\crefname{claim}{Claim}{Claims}
\crefname{equation}{}{}  
\Crefname{equation}{Equation}{Equations}

\DeclareMathOperator{\cost}{cost}
\DeclareMathOperator{\dev}{dev}

\newcommand{\OPT}{\operatorname{OPT}}

\newcommand{\Cs}{\mathcal{C}}

\newcommand{\support}{\mathrm{supp}}
\newcommand{\Kappa}{\mathrm{K}}

\newcommand{\lambdamax}{\lambda^{+}(\apdmatrices)}
\newcommand{\lambdamin}{\lambda^{-}(\apdmatrices)}
\newcommand{\AK}{{\apdmatrices}}
\newcommand{\error}{\mathcal{V}}

\newcommand{\coresetX}{\tilde{X}}

\newcommand{\coresetOmega}{\tilde{\Omega}}

\newcommand{\coresetC}{\tilde{C}}
\newcommand{\coresetinstance}{\tilde{I}}
\newcommand{\extension}{g}

\newcommand{\Ecc}{\mathcal{E}}

\newcommand{\wcaa}{$\mathbb{WCA}$ assignment\xspace}
\newcommand{\wcac}{$\mathbb{WCA}$ clustering\xspace}
\newcommand{\wcc}{$\mathbb{WC}$ clustering\xspace}
\newcommand{\wca}{$\mathbb{WC}$ assignment\xspace}

\newcommand{\apd}{\mathcal{P}}

\newcommand{\apdmatrices}{\mathcal{A}}
\newcommand{\apdmatrix}{A}
\newcommand{\apdweight}{\gamma}
\newcommand{\apdweights}{\Gamma}
\newcommand{\apdsites}{S}
\newcommand{\apds}{s}

\DeclareMathOperator*{\argmin}{argmin} %
\DeclarePairedDelimiterX{\norm}[1]{\lVert}{\rVert}{#1}

\newcommand{\R}{\mathbb{R}}
\newcommand{\N}{\mathbb{N}}

\makeatletter
\newcommand{\leqnomode}{\tagsleft@true\let\veqno\@@leqno}
\newcommand{\reqnomode}{\tagsleft@false\let\veqno\@@eqno}
\makeatother

\pgfplotsset{compat=newest} 
\bibliography{references}

\begin{document}
	\selectlanguage{english}
	\title{On resolution coresets for constrained clustering}
	\author{Maximilian Fiedler, Peter Gritzmann, Fabian Klemm}
	
	\publishers{\vspace*{4ex}%
		\normalfont\normalsize%
		\parbox{0.8\linewidth}{%
			\textbf{Abstract.} 
Specific data compression techniques, formalized by the concept of coresets, proved to be powerful for many optimization problems. In fact, while tightly controlling the approximation error, coresets may lead to significant speed up of the computations and hence allow to extend algorithms to much larger problem sizes. The present paper deals with a weight-balanced clustering problem from imaging in materials science. Here, the class of desired coresets is naturally confined to those which can be viewed as lowering the resolution of the data. Hence one would expect that such resolution coresets are inferior to unrestricted coreset. We show, however, that the restrictions are more than compensated by utilizing the underlying structure of the data. In particular, we prove bounds for resolution coresets which improve known bounds in the relevant dimensions and also lead to significantly faster algorithms practice.\newline

\vspace{-.2cm}

\vspace{-.2cm}

		}	
	}
	
	\maketitle


\section{Introduction}\label{sec:introduction}
Many challenging practical problems require the solution of weight-balanced clustering problems on huge data sets. Fields of applications include agriculture (consolidation of farmland, see e.g. \cite{BBG14}), transportation (air cargo prediction, \cite{BG20}), medicine (drug response rate analysis, see \cite{BG21}, \cite{SWJFCJB21}) or political sciences (electoral districting, \cite{BGK17}).

The present paper is more specifically motivated by the goal to provide highly efficient yet provably precise mathematical tools for studying the grain growth and structure of polycrystalline materials. In fact, \cite{philmag} applied constrained clustering formulations in order to compute generalized balanced power diagrams (GBPD) for concise representations of grain maps. While according to several independent studies, including \cite{sedivy16, sedivy17,Spettl2016}, the GBPD-algorithm seems to reflect the physical principles of forming polycrystals quite well, its computational cost exceed reasonable, let alone real-time requirements for large but practically relevant data sets. More specifically, dynamic simulations of grain growth at a small temporal resolution quickly drive (implementations based on the uncompressed data of) the method over the limit. Hence, currently, such simulations are typically based on heuristics, see, in particular, \cite{Lyckegaard2011, teferra18} and \cite{AFGK22b}. 

As proved in \cite{FG22} and experimentally verified in \cite{AFGK22a}, significant acceleration can already be obtained by the use of \enquote{general-purpose} coresets for constrained clustering. Such general coresets work on arbitrary finite data sets; hence ignore the special pixel- or voxel-based structure of the data in our specific and other image processing applications. 

The present paper addresses the question whether the resolution based image structure can be utilized to improve the known best bounds for (deterministic) coresets in spite of the additional requirement that such coresets should maintain a Cartesian product structure and hence can simply be interpreted as lower resolution images. 
We prove that \emph{resolution coresets} are indeed smaller in the relevant dimensions. Phrased differently, we show that even when applied at a relatively low resolution, the clustering techniques will still produce near-optimal weight-constrained clusterings. In fact, our bounds on the size of resolution coresets prove that they can be used to safely thin out the given data beyond what has been shown before, and hence put experimental observations of \cite{AFGK22a} on a firm theoretical ground.

The paper is organized as follows. First, \cref{sec:preliminaries_resolution_coreset} formally introduces the problems and states our main results. \Cref{sec:resolution_coreset_prelimiaries}  describes relevant concepts and provides some required preliminaries. Sections \ref{sec:proof1D}, \ref{sec:proof-d} and \ref{sec:approx-resolution-proof} then study resolution coresets in detail and prove our main results. \cref{sec:conclusion} concludes with final remarks.

\section{Motivation, notation, and main results}\label{sec:preliminaries_resolution_coreset}

We will now provide the relevant general notation, introduce the problems and state our main results. We begin with a brief nontechnical description of our \enquote{running application} of grain mapping. Formal details will follow later.

\subsubsection*{Grain mapping}\label{sec:grain}
Mathematical models play an important role for providing tools which can be used to understand the grain growth and structure of polycrystalline materials. Mathematically, grain maps are dissections of the occupied spatial region into monocrystalline cells.  Here we focus on a specific approach, introduced in \cite{philmag}, for efficiently describing such dissections. In fact, parameters characterizing a diagram representation of the grain map are determined as the solutions of an optimization problem. In effect, the points which represent the voxels of the desired grain image at a given resolution are clustered in such a way that certain characteristics of the resulting dissection coincide with available tomographic measurements.

A natural question is how low a resolution can be chosen while still keeping a high accuracy of representation. As the underlying resolution determines the dimension of the clustering problem, this is of utmost importance for practical computations as it may determine whether desired clusterings can actually be computed in practice.

\subsubsection*{Resolution}\label{sec:resolution}

In the following, we focus on pixel- or voxel-based images. While we are mainly interested in the dimensions $d=2,3$, we will give results in arbitrary dimensions. Hence we will use \emph{voxel} as generic term, independently of the dimension $d \in \N$ of space. Typically, we consider an image in $[0,1]^d$. Each axis is partitioned into intervals of equal length, and each voxel is the Cartesian product of one such interval on each axis. 

Formally, let $\bm\rho=(\rho_1,\ldots,\rho_d) \subset \N^d$ be a vector. We consider the partitioning of the $d$-dimensional cube $[0,1]^d$ into uniform boxes of size $2^{-\rho_1} \times \ldots \times 2^{-\rho_d}$. Then we obtain a data set $X(\bm \rho)$ (which will later be clustered into sets which represent the gains) by replacing each box by its centroid, i.e., $X(\bm \rho)$ consists of all points
\begin{equation*}
x=x_{j_1,\ldots,j_d} = \left(
\begin{array}{c}
\displaystyle \frac{1}{2^{\rho_1+1}} + \frac{j_1-1}{2^{\rho_1}} \\
\vdots\\
\displaystyle\frac{1}{2^{\rho_d+1}} + \frac{j_d-1}{2^{\rho_d}}\\
\end{array}
\right)
= \left(
\begin{array}{c}
\displaystyle \frac{-1}{2^{\rho_1+1}} + \frac{j_1}{2^{\rho_1}} \\
\vdots\\
\displaystyle\frac{-1}{2^{\rho_d+1}} + \frac{j_d}{2^{\rho_d}}\\
\end{array}
\right),
\end{equation*}
where, using the notation $[\ell]=\{1,\ldots,\ell\}$ for $\ell\in \N$,
\begin{equation*}
j_1\in [2^{\rho_1}], \ldots, j_d\in [2^{\rho_d}].
\end{equation*}
Note that 
\begin{equation*}
|X(\bm\rho)|= 2^{\|\bm\rho\|_1} = 2^{\sum_{i=1}^d \rho_i}.   
\end{equation*} 
Of course, the voxels of the image are the boxes
\begin{equation*}
x_{j_1,\ldots,j_d}
+ \frac{1}{2^{\rho_1+1}} [-1,1] \times \ldots \times \frac{1}{2^{\rho_d+1}} [-1,1]
\end{equation*}
of volume 
\begin{equation*}
   \nu=\nu (\bm \rho)=\frac{1}{2^{\rho_1}} \cdot \ldots \cdot \frac{1}{2^{\rho_d}}.
\end{equation*}

The spatial resolution of such a discretization of $[0,1]^d$ in each coordinate $i$ is $2^{\rho_i}$. For simplicity, we speak of the resolution $\bm{\rho}$.
In the following we will refer to point sets of the form $X(\bm \rho)$ as \emph{Cartesian point sets} or \emph{resolution set}. 

The paper addresses the error which occurs if we base the optimization on a coarser resolution $\bm \tau= (\tau_1,\ldots,\tau_d)$. We will adopt the following notation: $\bm\rho$ is the given \emph{benchmark resolution} with which we will compare all results for \emph{coarser resolutions} $\bm \tau$, i.e., $\bm \tau \le \bm \rho$. (Inequalities involving vectors are always meant componentwise.) 

\subsubsection*{Weight-constrained clustering}\label{sec:clusterings}
Next we introduce the constrained clustering problem. While we are specifically interested in Cartesian point sets the following description is more general.

Given $x_1,\ldots,x_n\in \R^d$ and associated weights $\omega_1,\ldots,\omega_n\in (0,\infty)$, let $(X,\Omega)$ denote the weighted data set of the tuples $(x_j,\omega_j)$ for $j\in [n]$. Further, let $k\in \N\setminus \{1\}$.

A $k$-\emph{clustering} (or, if $k$ is clear, simply a \emph{clustering}) $C$ of $(X,\Omega)$ is a vector
$$
C = (C_1,\ldots,C_k)=(\xi_{11},\ldots, \xi_{1n}, \ldots, \xi_{k1}, \ldots, \xi_{kn}) \in [0,1]^{kn}
$$ 
whose components $\xi_{ij}$ specify the fraction of $x_j$ that is assigned to the \emph{$i$th cluster} $C_i=(\xi_{i1},\ldots, \xi_{in})$. The set of all such clusterings will be denoted by $\Cs(k, X,\Omega)$.
If all $\xi_{ij}$ are integer, i.e., $\xi_{ij}\in \{0,1\}$, the clustering is called \emph{integer}.

Suppose that approximations $\kappa_1,\ldots,\kappa_k$ of the desired volumes of the cells are given; collected in the family $\Kappa=\{\kappa_1,\ldots,\kappa_k\}$. Then, we also demand that the \emph{weight} 
$$\omega(C_i)= \sum_{j=1}^{n} \xi_{ij}\omega_j
$$
of each cluster $C_i$ is close to $\kappa_i$. More precisely, 
for given $\epsilon_{i} \in (0,\infty)$, we require that the weak weight-constraints
\begin{equation*}
 (1-\epsilon_{i})\kappa_{i}  \le w(C_i) \le (1+\epsilon_{i})\kappa_{i} \qquad  (i \in [k])
\end{equation*}
hold. 
In the specific situation of grain mapping, $X=X(\bm{\rho})$, all weights are identical, in fact,
$\omega_i=\omega_i(\bm \rho)=\nu(\bm \rho)$ and, thus, the cluster weights of all integer clusterings are multiples of $\omega_i(\bm \rho)$. Hence with measurement errors $\epsilon_i< \nicefrac{1}{2}$ the weak weight-constraints are equivalent to the strong \emph{weight-constraints}
\begin{equation*}
    w(C_i) = \kappa_{i} \qquad (i \in [k])
\end{equation*}
As the measurement errors are generally determined by the benchmark resolution anyway we will in the following (just for the simplicity of exposition) always require this stronger condition. The set of all such \emph{weight-constrained clusterings} will be denoted by  $\Cs_\Kappa(k, X,\Omega)$. 
Note that, $\Cs_\Kappa(k, X,\Omega)\ne \emptyset$ if and only if
\begin{equation*}
    \sum_{i=1}^k \kappa_i = \sum_{j=1}^n \omega_j,
\end{equation*}
and we will assume that this is the case. 
As it is well-known (see e.g. \cite{BG12}), when all weights are $1$ or, equivalently, are identical, there always exist optimal clusterings $C\in \Cs_\Kappa(k, X,\Omega)$ which are integer.

The quality of a clustering will be measured in terms of deviations from $k$ different \emph{sites} $s_1,\ldots,s_k$, with respect to \emph{ellipsoidal norms}. More precisely, for $i\in [k]$, let  $A_{i}\in \R^{d\times d}$ be a positive definite symmetric matrix, and let $\norm{\,\cdot\,}_{A_i}$ denote the associated norm, i.e.
$\norm{x}_{A_i}=\sqrt{x^TA_ix}$ for any $x\in \R^d$. Further, we set $\apdsites=\{s_1,\ldots,s_k\}$ and $\apdmatrices=\{A_{1},\ldots,A_{k}\}$. Let us point out that $\apdmatrices$ is a family, i.e., repetitions are allowed. 

The \emph{cost} of a clustering $C$ with respect to $\apdsites$ and $\apdmatrices$ is defined as
\begin{equation*}
	\cost_\AK(X,C,\apdsites) = 
	\displaystyle \sum_{i=1}^{k}\sum_{j=1}^{n} \xi_{ij} \omega_j \norm{x_{j}-s_{i}}_{\apdmatrix_i}^{2}. 
\end{equation*}

A case of special relevance is that with all matrices being the unit matrix $E_d$. Then $\mathcal{E}=\{E_d,\ldots,E_d\}$ refers to the classic Euclidean \emph{least-squares} objective, in which each cluster norm is $\norm{x}_2=\norm{x}_{E}$. In the realm of grain mapping, this is called the \emph{isotropic case} while the general situation models the \emph{anisotropic case}. In the isotropic situation we drop the subscript and simply write  $\cost(X,C,\apdsites)$ rather than $\cost_\mathcal{E}(X,C,\apdsites)$.

If no sites $\apdsites$ are given, we measure the cost of a clustering $C$ with respect to the \enquote{best sites} (see \cref{rem:batching-error}), i.e., the cluster \emph{centroids}
\begin{equation*}
    c_i = \frac{1}{\omega(C_i)}\sum_{j=1}^n \xi_{ij} \omega_j x_j,
\end{equation*}
and define
\begin{equation*}
    \cost_\AK(X,C) = 
	\displaystyle \sum_{i=1}^{k}\sum_{j=1}^{n} \xi_{ij} \omega_j \norm{x_{j}-c_{i}}_{\apdmatrix_i}^{2}. 
\end{equation*}

Grain volumes $\Kappa$, sites $\apdsites$ and matrices $\apdmatrices$ can often be measured directly (with sites approximating the centroids). Different norms then reflect measured knowledge about the moments of the grains in materials, see \cite{philmag}. In the (rare) cases that grain scans are explicitly available the volumes $\Kappa$ can be derived directly, sites can be computed as centroids and the matrices can be estimated as the inverse of covariance matrices. However, both, $\apdsites$ and $\apdmatrices$ may also be subject to optimization to reduce the classification error; see \cite{AFGK22a}. Here we are dealing with the more standard situation, which leads to the mathematical problems of finding optimal weight-constrained anisotropic assignments or clusterings:
\begin{itemize}
 \item[] \emph{\wcaa}: \\ Given $k,X,\Omega,\apdmatrices,\Kappa,S$, find $C\in \Cs_\Kappa(k,X,\Omega)$ that minimizes $\cost_\AK(X,C,S)$.
 \item[] \emph{\wcac}: \\ Given $k,X,\Omega,\apdmatrices,\Kappa$, find $C\in \Cs_\Kappa(k,X,\Omega)$ that minimizes $\cost_\AK(X,C)$.
\end{itemize}
We say that $(k,X,\Omega,\apdmatrices,\Kappa,S)$ or  $(k,X,\Omega,\apdmatrices,\Kappa)$ is an instance of \wcaa or \wcac, respectively. If we are in the Euclidean case, i.e., when $\apdmatrices = \Ecc = \{E_d,\ldots,E_d\}$, we just speak of \emph{weight-constrained clustering (\wcc)} or \emph{assigment (\wca}) and often drop $\apdmatrices$ from the instance.

\subsubsection*{Coresets}\label{sec:coresets}
The notion of a coreset is used to classify the effect of data reductions on the quality of solutions of hard optimization problems. As it is well-known, even unconstrained least-squares clustering is NP-hard, \cite{Dasgupta2008, Aloise2009}, and the NP-hardness persists even if the dimension is fixed to $d=2$, \cite{Mahajan2012}. Also, by \cite{Awasthi2015}, no PTAS in both, $d$ and $k$, exists. Of course, when $\apdsites$ is fixed, the computation of
\begin{equation*}
    \cost_\AK(X,\apdsites)=\min_{C\in \Cs_\Kappa(k, X,\Omega)} \cost_\AK(X,C,\apdsites)
\end{equation*}
boils down to a linear program and can hence be carried out in polynomial time. But even then, typical real-world instances in materials science are so large that they can be solved only on very sparse grids. Therefore, one is aiming at a significant speed up by resorting to a \enquote{compressed} set $(\tilde{X},\tilde{\Omega})$ without sacrificing much of the clustering quality. 

More precisely, let $\epsilon \in (0,\nicefrac{1}{2}]$ and $\delta \in [1,\infty)$. Then $(\coresetX,\coresetOmega)$ is an \emph{$(\epsilon,\delta)$-coreset} for $(k,X,\Omega,\Kappa)$ if there exists a mapping $\extension:\Cs_\Kappa(k,\coresetX,\coresetOmega) \rightarrow \Cs_\Kappa(k,X,\Omega)$, called \emph{extension}, and real constants $\Delta^{+},\Delta^{-}$, referred to as \emph{$\Delta$-terms} or \emph{offsets}, with $0\le \Delta^{+} \le \delta \cdot \Delta^{-}$ such that the following two conditions hold for all sets $S$ of $k$ sites and clusterings $\coresetC\in  \Cs_\Kappa(k,\coresetX,\coresetOmega)$:
	\begingroup\leqnomode
	\begin{alignat}{3}
	(1-\epsilon)\cost_\AK(X,\extension(\coresetC),S) &\le \cost_\AK(\coresetX,\coresetC, S) + \Delta^{+},  && \tag{a} \label{def:property_a} \\
	\cost_\AK(\coresetX,S) + \Delta^{-} &\le (1+\epsilon)\cost_\AK(X, S). \;  &&  \tag{b} \label{def:property_b}
	\end{alignat}
	\endgroup	
If the two $\Delta$-terms coincide, i.e. $0\le \Delta^{+} = \Delta^{-}$, we can choose $\delta=1$, and we speak of an \emph{$\epsilon$-coreset} then. Note that, in this situation, for $\coresetC^*\in \argmin \{ \cost_\AK(\coresetX,\coresetC,S):  \coresetC\in \Cs_\Kappa(k,\coresetX,\coresetOmega)\}$, 
\begin{equation*}
(1-\epsilon)\cost_\AK(X, S)\le (1-\epsilon)\cost_\AK(X,\extension(\coresetC),S) 
 \le (1+\epsilon)\cost_\AK(X, S).
\end{equation*}
Hence, if the extension $g$ can be computed efficiently, solving the assignment problem on the coreset $(\coresetX,\coresetOmega)$ leads to an approximate solution on the original data set $(X,\Omega)$.
It is less obvious but still true that the above general definition of coresets indeed captures the intuition behind the concept that an approximate solution of the clustering problem on the coresets leads to an approximate solution on the original data set $(X,\Omega)$. As a service to the reader this result of \cite[Thm. 3.5]{FG22} will be restated in full generality (as \cref{prop:approximation_using_coreset}) in \cref{sec:approx-resolution-proof} (where we use a special case in the proof of \cref{th:approx-resolution} anyway).

Coresets for unconstrained least-squares clustering have been studied intensively, see \cite{Har-Peled2007,Feldman2013,FSS20,Fichtenberger2013,Sohler2018,Bachem2017a}. In particular, \cite{Har-Peled2007} constructed coresets of size $\mathcal{O}\bigl(\frac{k^3}{\epsilon^{d+1}}\bigr)$ for unconstrained least-squares clustering which \enquote{live} on $\mathcal{O}(k)$ pencils. Subsequently, \cite{FG22} improved and generalized their construction, leading to \emph{pencil coresets} of size only $\mathcal{O}\bigl(\frac{k^2}{\epsilon^{d+1}}\bigr)$ even for weight-constrained anisotropic clustering. In the following more precise statement of this result, $\lambdamax$ and $\lambdamin$ denote the largest and smallest eigenvalue of all matrices in the set $\apdmatrices$. 

\begin{proposition}[ {\cite[Thm. 2.3]{FG22}} ] 
\label{prop:smaller_coreset}
	For any instance $(k,X,\Omega,\Kappa,\apdmatrices)$ of \wcac, 
	$\epsilon\in (0,\nicefrac{1}{2}]$ and 
	$$\delta \ge \frac{\lambdamax}{\lambdamin},$$
	there is an $(\epsilon,\delta)$-coreset $(\coresetX,\coresetOmega)$ with 
	$$|\coresetX| \in \mathcal{O}\left(\frac{k^2}{\epsilon^{d+1}}\right).$$
\end{proposition}

While \cref{prop:smaller_coreset} gives the currently best bound for (deterministic) coresets for \wcac, it has still some severe drawbacks which limit its use in image segmentation problems like grain mapping. First, the available coresets may still be very large. For instance, for $k=100$, $d=3$, and $\epsilon= 0.001$, we have
\begin{equation*}
\frac{k^2}{\epsilon^{d+1}}= \frac{10^4}{10^{-12}}= 10^{16}.
\end{equation*}
Also, note that the asymptotic $\mathcal{O}$-statement for the size of $\coresetX$ conceals the constant which is also relevant.
As another crucial downside, the constructed pencil coresets completely ignore the grid structure of $X(\bm \rho)$ that is present in many image segmentation problems including our running application of grain mapping. 

\subsubsection*{Resolution coresets}\label{sec:resolution-coresets}

In the following we restrict the original data $(X,\Omega)$ to the \emph{resolution setting}
\begin{equation*}
X=X(\bm \rho)=\{x_1,\ldots,x_n\}, \qquad     
\omega_1=\ldots = \omega_n= \nu(\bm \rho)=\nu,
\end{equation*}
for a fixed resolution $\bm \rho$, and also assume that 
\begin{equation*}
\kappa_1,\ldots,\kappa_k \in \nu\N, \qquad n\nu= \sum_{i=1}^k \kappa_i.
\end{equation*}
We are interested in coresets $(\coresetX,\coresetOmega)$ which carry the same structure, i.e., for some resolution $\bm \tau$ with $\bm \tau \le \bm \rho$,
\begin{equation*}
\coresetX=X(\bm \tau)=\{\tilde{x}_1,\ldots,\tilde{x}_{\tilde{n}}\}, \quad     
\coresetOmega=\coresetOmega(\bm \tau)=\bigl\{\tilde{\omega}_1,\ldots, \tilde{\omega}_{\tilde{n}}\bigr\}
\end{equation*}
with 
\begin{equation*}
\tilde{\omega}_1=\ldots = \tilde{\omega}_{\tilde{n}}=\nu(\bm \tau)= 2^{-\tau_1}\cdot \ldots \cdot 2^{-\tau_d}.
\end{equation*}

Also, due to the Cartesian product structure of $X(\bm \rho)$ we restrict our considerations to the Euclidean case, i.e.,
\begin{equation*}
    \apdmatrices= \Ecc = \{E,\ldots,E\}.
\end{equation*}
While this is a (relevant) restriction to the isotropic case, \cref{th:approx-resolution} will show, however, that results for this setting also yield approximation results for the anisotropic case.

For this restricted setting we will study the question, how far we can reduce $\bm \tau$ while maintaining the coreset properties. More precisely, let $\epsilon \in (0,\nicefrac{1}{2}]$ and $\delta \in [1,\infty)$. Then $(\coresetX,\coresetOmega)$ is an \emph{$(\epsilon,\delta)$-resolution coreset} if $(\coresetX,\coresetOmega)$ is an \emph{$(\epsilon,\delta)$-coreset} and there exists a resolution $\bm \tau$ such that $\coresetX=\coresetX(\bm \tau)$ and $\coresetOmega=\coresetOmega(\bm \tau)$.
Note that such a Cartesian structure is beneficial for many applications that directly work on images, including superpixel segmentation or polycrystal representation. 

\subsubsection*{Main results}\label{sec:main-result}

On the one hand, the instances of our constrained clustering problem are strongly restricted, on the other hand only resolution coresets are permitted. Hence it is not clear whether the drawbacks of the more general result of \cref{prop:smaller_coreset} can be avoided in the decribed resolution setting at least in small dimensions (those which are relevant for grain mapping and other related image segmentation problems). As the following theorem shows, the answer is affirmative, even when we require that the offsets coincide.

\begin{theorem}\label{th:resolution-coreset}
	Let $\bm\rho \in \N^d$ $X=X(\bm\rho)$, $\Omega=\Omega(\bm\rho)$, and $\epsilon \in (0,\nicefrac{1}{2}]$. Then there is an $\epsilon$-resolution coreset $(\coresetX,\coresetOmega)$ for $(k,X,\Omega,\Kappa)$ of size 
	\begin{equation*}
	    |\coresetX| \le 2^{\nicefrac{8}{3}d}\left(\frac{k}{\epsilon^{\nicefrac{2}{3}}}\right)^d.
	\end{equation*}
\end{theorem}

Note that the size of the resolution coreset in \cref{th:resolution-coreset} is independent of $\bm{\rho}$. Thus the result shows that a higher resolution is often not needed to achieve a good clustering of image voxels. In fact, using a grid with less than $\nicefrac{7k}{\epsilon^{\nicefrac{2}{3}}}$ points on each coordinate axis results in a clustering whose cost is at most off the optimum by a factor of $1+\epsilon$.
In our running application from materials science, this is not only relevant for computing grain maps from few of their characteristic measurements but may also ease scanning processes used for verification. In fact,  \cref{th:resolution-coreset} allows to scan grains at a relatively low resolution without sacrificing much of the resulting quality and can hence be used to speed up scanning based imaging series of grain growth dynamics in practice.

Let us now compare the size of the resolution coreset to the bound given in \cref{prop:smaller_coreset}. Beginning with the asymptotics (and ignoring multiplicative constants), the bounds read
\begin{equation*}
  \frac{k^2}{\epsilon^{d+1}} \quad \mbox{(\cref{prop:smaller_coreset})}, \qquad  \frac{k^d}{\epsilon^{\nicefrac{2d}{3}}}\quad \mbox{(\cref{th:resolution-coreset})}.
\end{equation*}
Hence, the former is better in $k$ for $d\ge 3$ but always worse in $\epsilon$. 

In dimension $d=3$, the latter is better by a factor of $\epsilon^2 k$. Specifically, for values $k=100$ and $\epsilon= 0.001$, and when compared to the evaluation after \cref{prop:smaller_coreset}, we save a factor of $10^4$. Note that the bound for resolution coresets involves an additional constant of $256$, while the $\mathcal{O}$-notation in \cref{prop:smaller_coreset} hides the constant required for constructing the required $\epsilon$-nets. Such constructions typically involve a large constant that is exponential in $d$; see e.g. \cite[Lemma 5.3]{Brieden1998a}.

Let us point out that, as we will see, the construction of \cref{th:resolution-coreset} is explicit, including the extension $g$. Hence, from a constrained clustering $\coresetC$ on the low resolution data set we obtain quickly a clustering $g(\coresetC)$ on the original data set. Let us finally point out that this is not just useful in the isotropic case but also for anisotropic clusterings. 

\begin{theorem}\label{th:approx-resolution}
	Let $I(\apdmatrices)=\bigl(k,X(\bm \rho),\Omega(\bm \rho),\apdmatrices,\Kappa,S\bigr)$ be an instance of \wcaa, $I(\Ecc)=\bigl(k,X(\bm \rho),\Omega(\bm \rho),\Ecc,\Kappa,S\bigr)$, and  $\epsilon \in (0,\nicefrac{1}{2}]$.
	Further, let $(\coresetX,\coresetOmega)$ be an $\nicefrac{\epsilon}{3}$-resolution coreset for $I(\Ecc)$, $\extension$ its extension, $\coresetC \in \Cs_\Kappa(k,\coresetX,\coresetOmega)$ and $C=g(\coresetC)$. 
	
	If $\coresetC$ is a $\gamma$-approximation for $\coresetinstance(\Ecc)=(k,\coresetX,\coresetOmega,\Ecc,\Kappa,S)$, then $C$ is a $(1+\epsilon)\gamma \frac{\lambdamax}{\lambdamin}$-approximation for $I(\apdmatrices)$.
\end{theorem}

\section{Some basics for the analysis of resolution coresets}\label{sec:resolution_coreset_prelimiaries}

In the following we will prepare the proof of our main result. We will outline its structure, provide some prerequisites, and set our new estimates into the more general perspective of known approaches.

\subsubsection*{Structure of the proof of \cref{th:resolution-coreset}}

The construction of resolution coresets can be viewed as a \enquote{uniform local neighborhood merger}. Hence, we can, in principle, evoke \cite[Theorem 2.2]{FG22} to obtain an upper estimate for the required coreset size. It turns out, however, that the general estimates are much weaker than what we are aiming at. 

Hence, in order to prove \cref{th:resolution-coreset} we need to make much stronger use of the specific underlying setting, the Cartesian structure of the point set $X$, the uniform weights, and the fact that, in the isotropic case, the objective function is separable with respect to the coordinates. While the former allow us to use local arguments, the latter permits, in  effect, to reduce the estimates to the $1$-dimensional case. 
In fact, indicating, as usual, the $t$th coordinate of a vector $y$ by $y_{t}$, we have
\begin{align*}
	\cost_\AK(X,C,\apdsites) &= 
	\displaystyle \sum_{i=1}^{k}\sum_{j=1}^{n} \xi_{ij} \omega_j \norm{x_{j}-s_{i}}_{2}^{2}=
	\sum_{i=1}^{k}\sum_{j=1}^{n} \xi_{ij} \omega_j \sum_{t=1}^{d} \bigr((x_{j})_{t}-(s_{i})_{t}\bigl)^{2}\\
	&= \sum_{t=1}^{d}\left( \sum_{i=1}^{k}\sum_{j=1}^{n} \xi_{ij} \omega_j \bigr((x_{j})_{t}-(s_{i})_{t}\bigl)^{2}\right).
\end{align*}
We will, in the following, provide some preparatory results, then, in \cref{sec:proof1D}, conduct the computations and derive all estimates needed for the proof of \cref{th:resolution-coreset} in the $1$-dimensional case, and subsequently extend the result to arbitrary dimensions in \cref{sec:proof-d}. The general reasoning follows in part that of \cite{FG22}. We can, however, only use very few results directly while others, will have to be substantially modified. Yet most parts of the proofs are independent of \cite{FG22} and rely heavily on the specific structure of the present setting.

\subsubsection*{Merging function}

Let $\bm\rho=(\rho_1,\ldots,\rho_d) \in \N^d$, and $\bm\tau=(\tau_1,\ldots,\tau_d) \in \N^d$ with $\bm \tau \le \bm \rho$. We regard $\bm \rho$ as the fixed benchmark resolution while, at this point, $\bm \tau$ is a parameter vector which is adjusted later to specify the resolution of the desired coreset. As before,
\begin{equation*}
  X=X(\bm\rho), \quad n=|X|,\quad \Omega=\Omega(\bm\rho), \quad 
    \tilde{X}=X(\bm\tau), \quad \tilde{n}=|\tilde{X}|, \quad \tilde{\Omega}=\Omega(\bm\tau),
\end{equation*}
with weights 
\begin{equation*}
    \omega_j= \nu=2^{-\rho_1}\cdot \ldots \cdot 2^{-\rho_d}\quad \bigl(j\in [n]\bigr), \qquad 
    \tilde{\omega}_{q}= 2^{-\tau_1}\cdot \ldots \cdot 2^{-\tau_d}\quad \bigl(q\in [\tilde{n}]\bigr).
\end{equation*}
 We will interpret the coarsening of the resolution from $\bm \rho$ to $\bm \tau$ in terms of a neighborhood merger defined by a \emph{merging function}
 \begin{equation*}
    p: X  \rightarrow \tilde{X}. 
\end{equation*}
Here, $p$ is simply given by
\begin{equation*}
\left(
\begin{array}{c}
\displaystyle \frac{1}{2^{\rho_1+1}} + \frac{j_1-1}{2^{\rho_1}} \\
\vdots\\
\displaystyle\frac{1}{2^{\rho_d+1}} + \frac{j_d-1}{2^{\rho_d}}\\
\end{array}
\right) \quad \mapsto \quad 
\left(
\begin{array}{c}
\displaystyle \frac{1}{2^{\tau_1+1}} + \frac{q_1-1}{2^{\tau_1}} \\
\vdots\\
\displaystyle\frac{1}{2^{\tau_d+1}} + \frac{q_d-1}{2^{\tau_d}}\\
\end{array}
\right)
\end{equation*}
for
\begin{equation*}
j_t=r_t+(q_t-1)2^{\rho_t-\tau_t}, \quad r_t\in [2^{\rho_t-\tau_t}], \quad q_t\in  [2^{\tau_t}], \quad t\in [d].
\end{equation*}
Whenever this seems more convenient, we will regard $p$ as a function of the index set $[n]$, i.e., 
 \begin{equation*}
    p: [n] \rightarrow [\tilde{n}]. 
\end{equation*}
Note that, indeed,
\begin{equation*}
 \tilde{\omega}_{q}=  \sum_{j \in p^{-1}(q)} \omega_{j}
 \quad \bigl(q\in [\tilde{n}]\bigr).
\end{equation*}
In effect, we replace all points of $X$ in each \emph{batch}
\begin{equation*}
    B_q =\{x_j : j \in p^{-1}(q)\}
\end{equation*}
by $\tilde{x}_q$ which is carrying the total weight of all corresponding batch points. Observe that
\begin{equation*}
    |B_q|=\frac{\nu(\bm \tau)}{\nu(\bm \rho)}= 2^{\rho_1-\tau_1}\cdot \ldots \cdot 2^{\rho_d-\tau_d} = \frac{\tilde{\omega}_{q}}{\omega_j}.
\end{equation*}

Given a clustering $C \in \Cs_\Kappa(k,X,\Omega)$, the merging function $p$ gives rise to a clustering $\coresetC \in \Cs_\Kappa(k,\coresetX,\coresetOmega)$ defined by
\begin{equation*}
	\tilde{\xi}_{iq} = \frac{1}{\tilde{\omega}_{q}}\sum_{j \in p^{-1}(q)} \xi_{ij} \omega_{j} = \frac{\nu(\bm \rho)}{\nu(\bm \tau)}\sum_{j \in B_q} \xi_{ij} 
	\qquad \bigl(q \in [\tilde{n}],\, i\in [k]\bigr).
	\end{equation*}

Conversely, given $\coresetC\in\Cs_\Kappa(k,\coresetX,\coresetOmega)$, we obtain a clustering $C\in \Cs_\Kappa(k,X,\Omega)$ by setting 
\begin{equation*}
\xi_{ij} = \tilde{\xi}_{ip(j)} \qquad \bigl(j \in [n], i\in [k]\bigr).
\end{equation*}
This defines an extension 
\begin{equation*}
    \extension: \Cs_\Kappa(k,\coresetX,\coresetOmega) \rightarrow \Cs_\Kappa(k,X,\Omega)
\end{equation*}
that converts clusterings of $(\coresetX,\coresetOmega)$ into clusterings of $(X,\Omega)$. Note that the extension $\extension$ assigns each point $x_j$ of $X$ to clusters in the same way its representative $x_{p(j)}$ in $\coresetX$ is assigned to clusters. As we will prove later, the induced approximation error can be tightly controlled.

Although this is not an issue in our present context, let us point out in passing that coarsening the resolution generally excludes the integrality of clusterings. Of course, the constraint matrix is still totally unimodular but, unless the prescribed cluster weights $\kappa_i$ are divisible by $\nu(\bm \tau)$,  the set $\Cs_\Kappa(k,\coresetX,\coresetOmega)$ will not contain any integer clustering. While, by \cite[Cor. 2.3]{BG12}, the number of fractional variables $\tilde{\xi}_{i,q}$ that need to be accepted in optimal coreset clusterings is at most $2(k-1)$ the extension $g$ \enquote{spreads} any  such nonintegrality to all points in the corresponding batches $B_q$. 

\subsubsection*{Power diagrams}\label{sec:diagrams}

In the proof of \cref{th:resolution-coreset} we will make use of the close relation between constrained clustering and \emph{diagrams}. We will specifically use a result of \cite{BG12} for our isotropic setting; see, however, \cite{BGK17} and the handbook article \cite{GK17} for examples and properties of different classes of diagrams in the anisotropic case and additional pointers to the literature.

A \emph{power diagram} (PD) 
\begin{align*}
\apd= \apd(\apdsites,\Gamma) =(P_1,\ldots,P_k)
\end{align*}
in $\R^d$ is specified by $k$ \emph{sites} $\apdsites = \{\apds_1, \ldots, \apds_k\}$ in $\R^d$ and real \emph{sizes} $\Gamma = \{ \apdweight_1,\ldots, \apdweight_k\}$.
It is a dissection of $\R^d$ into polyhedral \emph{cells} defined by proximity, i.e., more precisely,
\begin{align*}
	P_i = \bigl\{ x \in \R^d:  \norm{x-s_i}_{2}^2 + \apdweight_i \leq   \norm{x-s_\ell}_{2}^2 + \apdweight_\ell ~~\forall \ell \in [k]\setminus \{i\}\bigr\}.
\end{align*}
Note that power diagrams generalize Voronoi diagrams (which result for $\gamma_1=\ldots=\gamma_k$). 

A clustering $C\in \Cs_\Kappa(k,X,\Omega)$ \emph{admits} a power diagram if there exists a power diagram $\apd$ such that 
\begin{equation*}
    \support(C_{i})=\{x_j \in X: \xi_{ij}> 0\} \subset P_{i} \qquad \bigl(i \in [k]\bigr).
\end{equation*} 
In this situation $\apd$ and $C$ are called \emph{compatible}. If, actually
\begin{equation*}
    \support(C_{i})=P_{i}\cap X \qquad \bigl(i \in [k]\bigr)
\end{equation*} 
then $\apd$ and $C$ are \emph{strongly compatible}. Note that if, in addition,  $C$ is integer then $\support(C_{i})$ lies in the interior of $P_{i}$ for each $i$, i.e., none of the point in $X$ lies on the boundary of any of the cells of $\apd$.

\begin{proposition}[ {\cite[Cor. 2.2]{BG12}} ] \label{prop:power-diagram}
 Any $C\in \Cs_\Kappa(k,X,\Omega)$ which attains $\cost(X,S)$ admits a power diagram. Further, there exists an optimal clustering $C^*\in \Cs_\Kappa(k,X,\Omega)$ and size parameters $\apdweights$ such that $C^*$ and the diagram $\apd=\apd(\apdmatrices, \apdsites, \apdweights)$ are strongly compatible.
\end{proposition}

\subsubsection*{Batch error}

For a nonempty subset $Y$ of $X$, let 
\begin{equation*}
    c(Y)= \frac{1}{|Y|}\sum_{x\in Y} x
\end{equation*}
denote its \emph{centroid}, and set
\begin{equation*}
\error (Y) = \sum_{x \in Y} \nu(\bm \rho) \norm{x-c(Y)}_{2}^{2}.
\end{equation*}
We will mostly (but not exclusively) apply this notion to batches $B_q$ with respect to the resolution $\bm \tau$, then use the notation $\error_{\bm\tau} (B_q)$ and speak of the \emph{batch error}.

Let us point out that the centroid $c(B_q)$ of a batch $B_q$ lies in $\tilde{X}$, and is actually the point resulting from the merging process. 

\begin{remark}\label{rem:centroid-merging}
Let $q\in [\tilde{n}]$, then $c(B_q)=\tilde{x}_q$.
\end{remark}

\begin{proof} 
For $t\in [d]$, let $q_t\in [2^{\tau_t}]$ such that $\tilde{x}_q=\tilde{x}_{q_1,\ldots,q_d}$. Then, using the abbreviation $\alpha_t=2^{\rho_t-\tau_t}$, we have for the $t$th coordinate $c(B_q)_{t}$ of the centroid $c(B_q)$
\begin{align*}
c(B_q)_{t}&= \frac{1}{|B_q|} \sum_{x\in B_q} c(B_q)_{t}
= \frac{\nu(\bm \rho)}{\nu(\bm \tau)} \sum_{x\in B_q} c(B_q)_{t}
=\frac{1}{\alpha_t} \sum_{r_t=1}^{\alpha_t}   
\left(\frac{-1}{2^{\rho_t+1}} + \frac{r_t+(q_t-1)\alpha_t}{2^{\rho_t}} \right)\\
&= \frac{1}{\alpha_t}\left(\frac{-\alpha_t}{2^{\rho_t+1}} + \frac{\alpha_t(\alpha_t+1)}{2^{\rho_t+1}}\right)+\frac{\alpha_t^2(q_t-1)}{2^{\rho_t}}\\
&= \frac{\alpha_t (2q_t-1)}{2^{\rho_t+1}}= \frac{2q_t-1}{2^{\tau_t+1}} =  \frac{1}{2^{\tau_t+1}} + \frac{q_t-1}{2^{\tau_t}} = (\tilde{x}_q)_{t}.
\end{align*}
\end{proof}

\cref{rem:centroid-merging} shows that the reduction of the resolution from $\bm \rho$ to $\bm \tau$ can be viewed as batching. Accordingly, the batch error 
$\error_{\bm\tau} (B_q)$ quantifies the effect of the merging process. Note that, due to the uniform batching process, the batch error is independent of $q$ and we will therefore omit $B_q$ and simply write $\error(\bm \tau)=\error_{\bm \tau}(B_q)$.

\begin{lemma} \label{le:resolution_batch_error}
    We have
    \begin{equation*}
        \error(\bm \tau) =  \frac{1}{12} \frac{1}{2^{\tau_1}} \cdot \ldots \cdot \frac{1}{2^{\tau_d}} \sum_{t=1}^d \left(\frac{1}{2^{2\tau_t}}- \frac{1}{2^{2\rho_t}}\right).
    \end{equation*}
\end{lemma}

\begin{proof}
Let $B=B_1$ denote the batch with centroid $\bigl(2^{-(\tau_1+1)},\ldots,2^{-(\tau_d+1)}\bigr)$. Then
\begin{align*}
\error(\bm \tau)&=\error(B)=\sum_{x \in B} \nu(\bm \rho) \norm{x-c(B)}_{2}^{2} 
= \nu(\bm \rho) \sum_{t=1}^d \sum_{x \in B} \bigl(x_{t}-c(B)_{t}\bigr)^{2}\\
&= \nu(\bm \rho) \sum_{t=1}^d \left(\prod_{\ell\in [d]\setminus \{t\}} 2^{\rho_\ell-\tau_\ell}\right) \sum_{r=1}^{2^{\rho_t-\tau_t}} \left(\frac{-1}{2^{\rho_t+1}}+\frac{r}{2^{\rho_t}}- \frac{1}{2^{\tau_t+1}} \right)^2\\
&= \nu(\bm \tau) \sum_{t=1}^d \frac{1}{2^{\rho_t-\tau_t}} \sum_{r=1}^{2^{\rho_t-\tau_t}} \left(\frac{-1}{2^{\rho_t+1}}+\frac{r}{2^{\rho_t}}- \frac{1}{2^{\tau_t+1}} \right)^2.
\end{align*}
Now, for $\rho=\rho_t$, $\tau=\tau_t$ and with $\alpha=2^{\rho-\tau}$,
    \begin{align*}
     \sum_{r=1}^{\alpha} &\left(\frac{-1}{2^{\rho+1}} + \frac{r}{2^\rho} - \frac{1}{2^{\tau+1}} \right)^2
     = \frac{1}{2^{2\rho+2}} \sum_{r=1}^{\alpha} \Bigl(2r-(\alpha+1)\Bigr)^2\\
     &= \frac{1}{2^{2\rho+2}} \Bigl(4 \frac{\alpha(\alpha+1)(2\alpha+1)}{6} - 4(\alpha+1) \frac{\alpha(\alpha+1)}{2}+\alpha (\alpha+1)^2\Bigr)^2\\
     &= \frac{1}{3}\frac{1}{2^{2\rho+2}} \alpha (\alpha^2-1)
     = \frac{1}{3} \frac{1}{2^{\rho+\tau+2}} (2^{2(\rho-\tau)}-1).
    \end{align*}
Hence,
\begin{align*}
\error(\bm \tau)&= \frac{1}{3} \nu(\bm \tau) \sum_{t=1}^d \frac{1}{2^{\rho_t-\tau_t}} \frac{1}{2^{\rho_t+\tau_t+2}} (2^{2(\rho_t-\tau_t)}-1)
= \frac{1}{3} \nu(\bm \tau) \sum_{t=1}^d \frac{1}{2^{2(\rho_t+1)}} (2^{2(\rho_t-\tau_t)}-1)\\
& = \frac{1}{12} \nu(\bm \tau) \sum_{t=1}^d \left(\frac{1}{2^{2\tau_t}}- \frac{1}{2^{2\rho_t}}\right),
\end{align*}
as claimed.
\end{proof}

In the next section we will also consider the error with respect to an arbitrary site. This can be handled with the aid of the following remark (which is folklore, can be found in \cite{BBG17}, and actually holds in greater generality than needed here, see \cite[Lemma 3.2 (a)]{FG22}). 

\begin{remark}\label{rem:batching-error}
 Let $Y$ be a nonempty subset of $X$, $s\in \R$. Then,
   \begin{equation*}
        \sum_{x\in Y} \nu(\bm \rho)  \|x-s\|_2^{2} = \error(Y) +  \nu(\bm \rho) \cdot |Y| \cdot\bigl\|c(Y)-s\bigr\|_2^{2},
   \end{equation*}
hence, in particular, the error on the left is minimized for $s=c(Y)$.
\end{remark}

\begin{proof} Since 
$\|x-s\|_2^{2}= \|x-c\|_2^{2} + 2(x-c)^T(c-s) +\|c-s\|_2^{2}$  and 
 \begin{equation*}
     \sum_{x\in Y} (x-c)^T(c-s) = \left(\Bigl(\sum_{x\in Y} x\Bigr)-|Y| c\right)^T(c-s) = 0,
 \end{equation*}
we have
 \begin{align*}
     \sum_{x\in Y} \nu(\bm \rho) \|x-s\|_2^{2} & = \left(\sum_{x\in Y} \nu(\bm \rho) \|x-c\|_2^{2}\right) + \nu(\bm \rho) \cdot |Y| \cdot \|c-s\|_2^{2} \\
     &= \error(Y) + \nu(\bm \rho) \cdot |Y| \cdot \|c-s\|_2^{2},
 \end{align*}
 which also implies the second assertion.
\end{proof}

\section{Proof of \cref{th:resolution-coreset}: The case $d=1$} \label{sec:proof1D}

While up to now all computations were exact for arbitrary dimension $d$ we will next derive estimates for the costs of clustering in dimension $1$. In \cref{sec:proof-d} these results will, in effect, be applied for each line parallel to any coordinate axis through points of $\coresetX$, and we will formulate them already with a view towards their later use.

Since, in this section, we are considering exclusively the $1$-dimensional situation we will adopt the simplified notation 
\begin{equation*}
   d=1, \quad  \rho={\bm \rho}, \quad \tau={\bm \tau}, 
\end{equation*}
and do not distinguish between the $1$-dimensional vectors $x_j$, $\tilde{x}_q$ and their coordinates $(x_j)_{1}$, $(\tilde{x}_j)_{1}$, respectively.

\subsubsection*{Best unconstrained least squares clustering}
We begin by computing the cost $\OPT_{\hat{k}}(X)$ of a \emph{best unconstrained least squares clustering} of the $1$-dimensional resolution set $X$ into 
\begin{equation*}
    \hat{k}=2^{\lceil \log(k)\rceil}
\end{equation*}
clusters. The following result will later be used as a lower bound for $\cost(X,C,S)$ for any choice of $C$ and $S$.

\begin{theorem}\label{thm:1D-clustering}
    Let $\gamma=\lceil \log(k)\rceil$, $\hat{k}=2^{\gamma}$, and let $\hat{C}=(C_1,\ldots,C_{\hat{k}})$ be an optimal unconstrained least-squares clustering of $X$. Then, each cluster contains the same number $2^{\rho-\gamma}$ of points of $X$, the centroids $c_i=c(C_i)$ of
    the $\hat{k}$ clusters are the equidistant points
\begin{equation*}
   c_{i}= \frac{1}{2\hat{k}} + \frac{i-1}{\hat{k}} = \frac{1}{2^{\gamma+1}} + \frac{i-1}{2^\gamma} \qquad (i\in [\hat{k}]),
\end{equation*}
and 
\begin{equation*}
    \OPT_{\hat{k}}(X)= \frac{1}{3} \frac{1}{2^{2(\rho+1)}} \bigl(2^{2(\rho-\gamma)}-1\bigr)
    = \frac{1}{3} \frac{1}{2^{2(\rho+1)}} \left(\frac{2^{2\rho}}{\hat{k}^2}-1\right).
\end{equation*}
\end{theorem}

\begin{proof} Since, by \cref{prop:power-diagram}, $\hat{C}$ admits a power diagram, the points in each cluster are consecutive. Without loss of generality we may also assume that $\hat{C}$ is integer. So, with $\beta_{\hat{k}+1}=n$, 
\begin{equation*}
    \alpha_i=|C_i|, \quad \beta_i = \sum_{\ell=1}^{i-1} \alpha_\ell, \quad 
    I_i=[\beta_{i+1}]\setminus [\beta_{i}] \qquad \bigl(i\in [\hat{k}]\bigr)
\end{equation*}
(and the standard interpretation of empty sums and $[0]=\emptyset$) we have for $j\in I_i$
\begin{equation*}
    x_j = \frac{1}{2^{\rho+1}} + \frac{j-1}{2^\rho} = \frac{\beta_i}{2^{\rho}} + \frac{1}{2^{\rho+1}} +\frac{j-1-\beta_i}{2^{\rho}} = \frac{\beta_i}{2^{\rho}} + x_{j-\beta_i}.
\end{equation*}
Hence
\begin{align*}
    c_i& = \frac{1}{\alpha_i} \sum_{j\in I_i} x_j 
    = \frac{\beta_i}{2^{\rho}} + \frac{1}{\alpha_i} \sum_{j=1}^{\alpha_i} x_j
    = \frac{\beta_i}{2^{\rho}} +\frac{1}{\alpha_i} \sum_{j=1}^{\alpha_i} \left(\frac{1}{2^{\rho+1}} + \frac{j-1}{2^\rho}\right)\\
    &= \frac{\beta_i}{2^{\rho}} + \frac{1}{\alpha_i}\left(\frac{\alpha_i}{2^{\rho+1}} + \frac{1}{2^\rho} \sum_{j=1}^{\alpha_i}(j-1)\right)
    = \frac{\beta_i}{2^{\rho}} + \frac{\alpha_i}{2^{\rho+1}}.   
\end{align*}
Therefore we obtain as $C_i$'s contribution to the cost of the clustering
 \begin{align*}
 \error(C_i) & = \sum_{j=1}^{\alpha_i} \nu (\rho) \Bigl(x_{j}-c_i\Bigr)^2
 = \nu(\rho)
 \sum_{j=1}^{\alpha_i} \left(\frac{1}{2^{\rho+1}} + \frac{j-1}{2^\rho}-\frac{\alpha_i}{2^{\rho+1}}\right)^2\\
 &=
 \frac{1}{2^{3\rho+2}}\sum_{j=1}^{\alpha_i} \Bigl((\alpha_i+1)-2j\Bigr)^2
 = \frac{1}{2^{3\rho+2}}\sum_{j=1}^{\alpha_i} \Bigl((\alpha_i+1)^2-4(\alpha_i+1)j+4j^2\Bigr)\\
 &=
\frac{1}{2^{3\rho+2}} \left(-\alpha_i(\alpha_i+1)^2+ \frac{2}{3} \alpha_i(\alpha_i+1)(2\alpha_i+1)\right)
=
 \frac{1}{3} \frac{1}{2^{3\rho+2}} \alpha_i(\alpha_i^2-1).
 \end{align*}
Since $\hat{C}$ is optimal the cluster cardinalities minimize $\sum_{i=1}^{\hat{k}} \alpha_i(\alpha_i^2-1)$ (as a function of the $\alpha_i$) under the constraints $\sum_{i=1}^{\hat{k}} \alpha_i=n$ and $\alpha_i\in \N\cup \{0\}$. Now note that the $i$th summand of the objective function is strictly increasing in $\alpha_i$. Hence, in the minimum we must have $\alpha_i \in \N$ for all $i \in [\hat{k}]$.

Let us now discard the integrality condition and consider the constrained nonlinear program
\begin{equation*}
\begin{array}{cccccccc}
 	&           & \multicolumn{4}{c}{\displaystyle \min\quad \sum_{i=1}^{\hat{k}} \alpha_i(\alpha_i^2-1)} && \\[.4cm]
 	& \qquad & & \displaystyle \sum_{i=1}^{\hat{k}} \alpha_i & = & n  &\quad & \\[.4cm]
 	 	& \qquad & & \alpha_{i} & \ge & 1  &\quad & \bigl(i \in [\hat{k}]\bigr)\\
 \end{array}
 \end{equation*}
in the real variables $\alpha_1,\ldots,\alpha_k$. As all occurring functions are continuously differentiable and the objective function is convex in the nonnegative orthant we can apply the Karush-Kuhn-Tucker Theorem and obtain the conditions 
\begin{align*}
    &(3\alpha_1^2-1,\ldots, 3\alpha_k^2-1) + \sum_{i=1}^{\hat{k}} \mu_i (-1,\ldots,-1) + \lambda (1,\ldots,1)=0,\\
    & \sum_{i=1}^{\hat{k}}\alpha_i=n, \qquad \alpha_1,\ldots, \alpha_{\hat{k}} \ge 1, \qquad
     \mu_1,\ldots, \mu_{\hat{k}} \ge 0, \qquad  \sum_{i=1}^{\hat{k}} \mu_i \alpha_i=0
\end{align*}
involving the Lagrange parameters $\mu_i$ and $\lambda$. Since $\alpha_i\ge 1$ and $\mu_i\ge 0$ for all $i\in [\hat{k}]$, the complementary slackness condition yields  $\mu_1=\ldots=\mu_{\hat{k}}=0$, and the first condition reads
\begin{equation*}
    (3\alpha_1^2-1+\lambda,\ldots, 3\alpha_{\hat{k}}^2-1+\lambda) = 0.
\end{equation*}
This implies that all $\alpha_i$ coincide. Thus we obtain
\begin{equation*}
    \alpha_1=\ldots=\alpha_{\hat{k}}=\nicefrac{n}{\hat{k}}= 2^{\rho-\gamma} \in \N, \qquad \beta_i=(i-1)2^{\rho-\gamma} \quad \bigl(i\in [\hat{k}]\bigr)
\end{equation*} 
and
\begin{equation*}
    c_i = \frac{\beta_i}{2^\rho}+ \frac{\alpha_i}{2^{\rho+1}}= \frac{1}{2^{\gamma+1}}+\frac{i-1}{2^\gamma} \qquad \bigl(i\in [\hat{k}]\bigr),
\end{equation*}
as claimed. Also, $\error(C_1)=\ldots= \error(C_{\hat{k}}) =\error(\gamma)$ and hence
\begin{equation*}
    \OPT_{\hat{k}}(X)= 2^\gamma \error(\gamma) 
    =  2^\gamma \cdot \frac{1}{3} \frac{1}{2^{3\rho+2}} 2^{\rho-\gamma}\bigl(2^{2(\rho-\gamma)}-1\bigr)
    = \frac{1}{3} \frac{1}{2^{2(\rho+1)}} \bigl(2^{2(\rho-\gamma)}-1\bigr),
\end{equation*}
which completes the proof.
\end{proof}

Note that the centroids (and hence, by \cref{rem:batching-error} best choices of sites) are exactly the points of the data set $X(\gamma)$ and thus the cluster error is the batch error $\error(\gamma)$. As a corollary we obtain the desired lower bound.

\begin{corollary}\label{cor:1-dim-bound}
    Let $S$ be a set of $k$ sites and $C\in \Cs_\Kappa(k, X,\Omega)$. Then  
    \begin{equation*}
        \cost (X,C,\apdsites) \ge \OPT_k(X) \ge \frac{1}{3} \frac{1}{2^{2(\rho+1)}} \left(\frac{2^{2(\rho-1)}}{k^2}-1\right).
    \end{equation*}
\end{corollary}

\begin{proof}
    Since
    \begin{equation*}
    k \le \hat{k}=2^{\lceil \log(k)\rceil}< 2^{\log(k)+1}\le 2k
\end{equation*}
and
        \begin{equation*}
    \OPT_{\hat{k}}(X) \le \OPT_k(X) \le \cost (X,C,\apdsites),
    \end{equation*}
\cref{thm:1D-clustering} implies
    \begin{equation*}
\frac{1}{3} \frac{1}{2^{2(\rho+1)}} \left(\frac{2^{2(\rho-1)}}{k^2}-1\right) < \frac{1}{3} \frac{1}{2^{2(\rho+1)}}  \left(\frac{2^{2\rho}}{\hat{k}^2}-1\right) \le \cost (X,C,\apdsites),
    \end{equation*}
as asserted.
\end{proof}

\subsubsection*{Coreset property \cref{def:property_a}}
Next, we prove an even stronger version of coreset property \labelcref{def:property_a}.

\begin{lemma} \label{le:coreset_property_a_line_resolution}
    Let $\Delta = 2^{\tau} \error(\tau)$, $S\subset \R$ a set of $k$ sites, $\coresetC\in\Cs_\Kappa(k,\coresetX,\coresetOmega)$, and $C = g(\coresetC)$, where $\extension: \Cs_\Kappa(k,\coresetX,\coresetOmega) \rightarrow \Cs_\Kappa(k,X,\Omega)$ is the standard extension defined in the previous section. Then
    \begin{equation*}
        \cost (X,C,S) = \cost (\coresetX,\coresetC, S) + \Delta
\end{equation*}
\end{lemma}

\begin{proof}
For $i\in [k]$, $j\in [n]$ and $q\in [\tilde{n}]$ let, as before, $\xi_{ij}$ and $\tilde{\xi}_{iq}$ denote the components of $C$ and $\coresetC$, respectively. Further recall, that the merging function maps each point of a batch $B_q$ to its centroid. Then, using \cref{rem:batching-error}, we have
\begin{align*}
    \cost &(X,C,S) 
    = \sum_{i=1}^k \sum_{j=1}^{n}  \nu(\rho) \cdot \xi_{ij}\cdot (x_j-s_i)^2 
     = \sum_{i=1}^k \sum_{q=1}^{\tilde{n}}\tilde\xi_{iq} \sum_{x_j \in B_q}   \nu(\rho) \cdot (x_j - s_i)^2 \\
   & = \sum_{i=1}^k \sum_{q=1}^{\tilde{n}} \tilde{\xi}_{iq} \left( \error(\tau) +  \nu(\rho)\cdot |B_q| \cdot(x_q-s_i)^{2} \right) \\
   &= \tilde{n} \error(\tau) + \sum_{i=1}^k \sum_{q=1}^{\tilde{n}} \tilde{\xi}_{iq}  \frac{2^{\rho-\tau}}{2^\rho}(x_q-s_i)^{2}
   = 2^\tau \error(\tau) + \sum_{i=1}^k \sum_{q=1}^{\tilde{n}} \nu(\tau)\tilde{\xi}_{iq}(x_q-s_i)^{2}\\
   & = \Delta+\cost (\coresetX,\coresetC, S),
\end{align*}
which proves the assertion.
\end{proof}

\subsubsection*{Coreset property \cref{def:property_b}: Setup and good indices}
Not surprisingly, the proof of coreset property \labelcref{def:property_b} is much more involved, and the rest of this section will be devoted to that. First, we observe, that it follows from a more general inequality.

\begin{remark}\label{rem:reduction}
Suppose that there exists a constant $\Delta$ such that the inequality
\begingroup\leqnomode
	\begin{alignat}{3}
	    \cost\bigl(\coresetX,p(C),S\bigr) + \Delta \le (1+\epsilon)  \cost\bigl(X,C,S\bigr)
  \tag{b${}^\prime$} \label{def:property_bprime}
	\end{alignat}
	\endgroup
holds for every set $S$ of $k$ sites and any integer clustering $C\in \Cs_\Kappa(k,X,\Omega)$ which admits a strongly compatible power diagram. Then, for every such $S$,
\begin{equation*}
     \cost\bigl(\coresetX,S\bigr) + \Delta \le (1+\epsilon)  \cost(X,S).
\end{equation*}
\end{remark}

\begin{proof}
    First note that, by \cref{prop:power-diagram}, there exists an integer clustering $C^*$ which admits a strongly compatible power diagram and attains $\cost(X,S)$. Then, of course,
\begin{equation*}
     \cost(\coresetX,S) \le  \cost\bigl(\coresetX,p(C^*),S\bigr) \qquad \text{and} \qquad
     \cost\bigl(X,C^*,S\bigr)  = \cost(X,S).
\end{equation*}
Hence, by assumption,
\begin{equation*}
     \cost\bigl(\coresetX,S\bigr) + \Delta \le \cost\bigl(\coresetX,p(C^*),S\bigr)+ \Delta \le 
     (1+\epsilon)  \cost(X,C^*,S) = (1+\epsilon)  \cost(X,S).
\end{equation*}
\end{proof}

According to \cref{rem:reduction}, it suffices to prove \cref{def:property_bprime}. So let, in the following, $S$ be a set of $k$ sites, let $C=(C_1,\ldots,C_k)\in \Cs_\Kappa(k,X,\Omega)$ be an integer clustering which admits the strongly compatible power diagram $\apd = (P_1,\ldots,P_k)$, and set $\coresetC=p(C)$.

We split the cost of the clustering $\coresetC$ of the resolution coreset into two components which reflect whether batches are fully contained in one of the cells of $\apd$ or split by the diagram.
More precisely, for each $q\in [\tilde{n}]$, let $i(q)$ denote the unique index $i\in [k]$ such that $x_q \in P_i$, and set
\begin{equation*}
    \tilde{I}^{(+)} = \bigl\{ q \in [\tilde{n}] : B_q \subset P_{i(q)} \bigr\}, \qquad \tilde{I}^{(-)} = [\tilde{n}] \setminus \tilde{I}^{(+)} .
\end{equation*}
The elements of $\tilde{I}^{(+)} $ and $\tilde{I}^{(-)} $ will be called \emph{good} and \emph{bad indices}, respectively, as the latter require significantly more care in terms of cost estimations. Note that for good indices $q$ the clusterings $C$ and $\coresetC$ \enquote{coincide} in the sense that
\begin{equation*}
     \tilde{\xi}_{iq}=\xi_{ij} \qquad \text{for all $i\in [k]$, $ j\in [n]$, $q\in \tilde{I}^{(+)}$ with $x_j\in B_q$}.
\end{equation*}
Also, as the points of $X$ in, both, the batches $B_q$ and the intervals $P_i$ are consecutive, we obtain the following bound.

\begin{remark}\label{rem:k-1}
\begin{equation*}
    |\tilde{I}^{(-)} | \le k-1.
\end{equation*}
\end{remark}

The cost of $C$ can be expressed in terms of the good and bad indices. In fact,
\begin{equation*}
        \cost(X,C,S)  = \sum_{i=1}^k \sum_{j=1}^{n} \nu(\rho) \cdot \xi_{ij}\cdot (x_j-s_i)^2 = \gamma(\tilde{I}^{(+)}) + \gamma(\tilde{I}^{(-)}),
\end{equation*}
where 
\begin{equation*}
    \gamma(\tilde{I}^{(\pm)} )= \sum_{i=1}^k \sum_{q \in \tilde{I}^{(\pm)} } \sum_{x_j \in B_q}  \nu(\rho) \cdot \xi_{ij}\cdot (x_j-s_i)^2.
\end{equation*}
Similarly, we use the abbreviations
\begin{equation*}
    \tilde{\gamma}(\tilde{I}^{(\pm)} )= \sum_{i=1}^k \sum_{q \in \tilde{I}^{(\pm)} } \nu(\tau) \cdot \tilde{\xi}_{iq}\cdot (\tilde{x}_q-s_i)^2
\end{equation*}
to denote the parts of $\cost(\tilde{X},\tilde{C},S)$ associated with $\tilde{I}^{(\pm)} $.

The cost for the good indices relates directly to the cost incurred by the corresponding original data points. In fact, by \cref{rem:batching-error} we have
\begin{align*}
    \gamma(\tilde{I}^{(+)} )&
    =  \sum_{i=1}^k \sum_{q \in \tilde{I}^{(+)} } \sum_{x_j \in B_q\cap C_i}  \nu(\rho) \cdot (x_j-s_i)^2\\
    & =  \sum_{i=1}^k \sum_{q \in \tilde{I}^{(+)} } \left( \error(B_q\cap C_i) + \nu(\rho) \cdot |B_q\cap C_i| \cdot (\tilde{x}_q-s_i)^2 \right)\\
    & = \sum_{i=1}^k \sum_{q \in \tilde{I}^{(+)} } \tilde{\xi}_{iq}\left( \error(\tau) + \nu(\tau) \cdot (\tilde{x}_q-s_i)^2 \right).
\end{align*}
Hence, we obtain for the good indices,

\begin{remark}\label{rem:good_indices_equality}
\begin{equation*}
 \gamma(\tilde{I}^{(+)} )
 = |\tilde{I}^{(+)} |\cdot \error(\tau) + \tilde{\gamma}(\tilde{I}^{(+)} ).
\end{equation*}
\end{remark}

\subsubsection*{Bad indices}

Now we deal with the indices in $\tilde{I}^{(-)} $. Naturally, the clustering error for the bad indices depends more specifically on the batch sizes. In the following, we assume that $\tau = \tau^*$ is chosen appropriately. In view of the extension of our results to the $d$-dimensional case in \cref{sec:proof-d} we show slightly more than what would otherwise be needed in this section. So, let $k^* \in \N$ such that $k\le k^*$ and set
\begin{equation*}
 \tau^* = \left\lceil \log\left(\frac{2^{\nicefrac{5}{3}}k^*}{\epsilon^{\nicefrac{2}{3}}}\right) \right\rceil.
\end{equation*}
Of course, we also assume that $\rho\ge \tau^*$ as before.

Next we deal with the \emph{total deviation}
    \begin{equation*}
    \dev(\tilde{I}^{(-)} )= \sum_{i=1}^k \sum_{q\in \tilde{I}^{(-)} } \sum_{x_j\in B_q} \xi_{ij} \cdot \nu(\rho) \cdot (\tilde{x}_q- x_j)^2
    \end{equation*}
caused by the bad indices. 
Let us first point out how $\dev(\tilde{I}^{(-)} )$ relates to the batch error.

\begin{remark}\label{rem:total-dev}
    \begin{equation*}
    \dev(\tilde{I}^{(-)} ) = |\tilde{I}^{(-)} | \cdot \error(\tau^*).
    \end{equation*}
\end{remark}

\begin{proof}
    By \cref{rem:centroid-merging} we have
\begin{align*}
    \dev(\tilde{I}^{(-)} ) & = \sum_{q\in \tilde{I}^{(-)} } \sum_{x_j\in B_q} \sum_{i=1}^k \xi_{ij} \cdot \nu(\rho) \cdot (\tilde{x}_q- x_j)^2
    = \sum_{q\in \tilde{I}^{(-)} } \sum_{x_j\in B_q} \nu(\rho) \cdot (\tilde{x}_q- x_j)^2\\
    &= \sum_{q\in \tilde{I}^{(-)} } \sum_{x_j\in B_q} \nu(\rho) \cdot (c(B_q)- x_j)^2 
    = \sum_{q\in \tilde{I}^{(-)} } \error( B_q)
    = |\tilde{I}^{(-)} | \cdot \error(\tau^*).
\end{align*}
\end{proof}

The next lemma provides an upper bound for the total deviation. It will employ \cref{le:resolution_batch_error} in its following $1$-dimensional form. 

\begin{remark} \label{rem:resolution_batch_error}
    \begin{equation*}
        \error(\tau) =  \frac{1}{12}\frac{1}{2^{\rho}}  \left(\frac{1}{2^{2\tau}}- \frac{1}{2^{2\rho}}\right)= \frac{1}{3}\frac{1}{2^{\rho}}  \frac{2^{2(\rho-\tau)}-1}{2^{\rho+\tau+2}}.
    \end{equation*}
\end{remark}

\begin{lemma} \label{le:bound_bad_indices}
    \begin{equation*}
    \dev(\tilde{I}^{(-)} ) \le \frac{\epsilon^2}{8}  \OPT_k(X) < \frac{\epsilon}{6} \cost(X,C,S).
    \end{equation*}
\end{lemma}

\begin{proof}
By \cref{rem:k-1,rem:total-dev,rem:resolution_batch_error},
\begin{align*}
    \dev(\tilde{I}^{(-)} ) = |\tilde{I}^{(-)} | \cdot \error(\tau^*) 
    \le (k-1) \cdot \frac{1}{3} \frac{1}{2^{2\rho+\tau^*+2}} \bigl(2^{2(\rho-\tau^*)}-1\bigr),
\end{align*}
and, by \cref{cor:1-dim-bound},
\begin{equation*}
 \OPT_k(X) \ge \frac{1}{3} \frac{1}{2^{2(\rho+1)}} \left(\frac{2^{2(\rho-1)}}{k^2}-1\right).
\end{equation*}
Thus 
    \begin{align*}
    \dev(\tilde{I}^{(-)} ) & - \frac{\epsilon^2}{8}  \OPT_k(X)
    \le \frac{1}{3} \frac{1}{2^{2(\rho+1)}} \left( \frac{k-1}{2^{\tau^*}} \Bigl(2^{2(\rho-\tau^*)}-1\Bigr) - \frac{\epsilon^2}{8} \Bigl(\frac{2^{2(\rho-1)}}{k^2}-1\Bigr) \right).
    \end{align*}
Setting 
\begin{align*}
     f(k)=\frac{k-1}{2^{\tau^*}} \Bigl(2^{2(\rho-\tau^*)}-1\Bigr) - \frac{\epsilon^2}{8} \Bigl(\frac{2^{2(\rho-1)}}{k^2}-1\Bigr)
    \end{align*}
we have for $k< k^*$
    \begin{align*}
     f(k+1)-f(k)= \frac{1}{2^{\tau^*}} \left(2^{2(\rho-\tau^*)} - 1\right) + \frac{\epsilon^2 \cdot 2^{2(\rho-1)}}{8}\left( \frac{2k+1}{k^2(k+1)^2} \right) \ge 0.
    \end{align*}
Hence $f(k)$ is increasing in $k$, and we have 
\begin{align*}
    \dev(\tilde{I}^{(-)} ) & - \frac{\epsilon^2}{8}  \OPT_k(X)
    \le \frac{1}{3} \frac{1}{2^{2(\rho+1)}} \left( \frac{k^*-1}{2^{\tau^*}} \Bigl(2^{2(\rho-\tau^*)}-1\Bigr) - \frac{\epsilon^2}{8} \Bigl(\frac{2^{2(\rho-1)}}{(k^*)^2}-1\Bigr) \right)\\
    &= \frac{1}{3} \frac{1}{2^{2(\rho+1)}} \left( 2^{2\rho} \Bigr(\frac{k^*-1}{2^{3\tau^*}} - \frac{\epsilon^2}{32(k^*)^2}\Bigr)
    - \Bigl( \frac{k^*-1}{2^{\tau^*}} - \frac{\epsilon^2}{8} \Bigr)\right),
    \end{align*}
and, to prove the first asserted inequality, it suffices to show that 
    \begin{align*}
    \frac{k^*-1}{2^{3\tau^*}} \le  \frac{\epsilon^2}{32(k^*)^2}, \qquad
    \frac{k^*-1}{2^{\tau^*}} \ge \frac{\epsilon^2}{8}.
    \end{align*}
To verify these inequalities, simply note that,
\begin{equation*}
   \frac{(k^*-1) (k^*)^2}{2^{3\tau^*}} \le \frac{(k^*-1)\epsilon^2}{2^5k^*} \le \frac{\epsilon^2} {32},
\end{equation*}
and, since $k^*\ge 2$ and $\epsilon \in (0,\nicefrac{1}{2}]$,
    \begin{equation*}
    \frac{k^*-1}{2^{\tau^*}} \ge \frac{k^*-1}{4k^*} \cdot \epsilon^{\nicefrac{2}{3}}
    = \frac{\epsilon^2}{8} \left(\frac{k^*-1}{k^*} \frac{2}{\epsilon^{\nicefrac{4}{3}}}\right)
    \ge \frac{\epsilon^2}{8}.
\end{equation*}
Finally note that $\nicefrac{\epsilon^2}{8} \le \nicefrac{\epsilon}{16} < \nicefrac{\epsilon}{6}$.
\end{proof}

The next lemma gives the desired bound for the bad indices. 

\begin{lemma} \label{le:inequality_bad_indices}
    \begin{align*}
          \tilde{\gamma}(\tilde{I}^{(-)} )  \le \gamma(\tilde{I}^{(-)} ) + \frac{5\epsilon}{6} \cost(X,C,S).
    \end{align*}
\end{lemma}

\begin{proof}
Let us remark that the assertion can also be inferred from the more general result in \cite[Lemma 4.3]{FG22}. As a service to the reader we give an explicit direct proof of what we need here. In fact,
\begin{align*}
    \tilde{\gamma}(\tilde{I}^{(-)} ) & =  \sum_{i=1}^k \sum_{q \in \tilde{I}^{(-)} } \nu(\tau) \cdot \tilde{\xi}_{iq}\cdot (\tilde{x}_q-s_i)^2 \\
    &=  \sum_{i=1}^k \sum_{q \in \tilde{I}^{(-)} } \sum_{x_j\in B_q} \nu(\rho) \cdot \xi_{ij}\cdot \bigl((\tilde{x}_q-x_j)+(x_j-s_i)\bigr)^2\\
    & = \dev(\tilde{I}^{(-)} ) + \gamma(\tilde{I}^{(-)} ) + 2 \sum_{i=1}^k \sum_{q \in \tilde{I}^{(-)} } \sum_{x_j\in B_q} \nu(\rho) \cdot \xi_{ij}\cdot (\tilde{x}_q-x_j)(x_j-s_i).
\end{align*}
Further, using the Cauchy-Schwartz inequality,
\begin{align*}
 \Biggl(\sum_{i=1}^k & \sum_{q \in \tilde{I}^{(-)} } \sum_{x_j\in B_q} \nu(\rho) \cdot \xi_{ij}\cdot (\tilde{x}_q-x_j)(x_j-s_i)\Biggr)^2 \\
 &\le \Biggl(\sum_{i=1}^k \sum_{q \in \tilde{I}^{(-)} } \sum_{x_j\in B_q} \nu(\rho) \cdot \xi_{ij}\cdot (\tilde{x}_q-x_j)^2\Biggr) \cdot 
 \Biggl(\sum_{i=1}^k \sum_{q \in \tilde{I}^{(-)} } \sum_{x_j\in B_q} \nu(\rho) \cdot \xi_{ij}\cdot (x_j-s_i)^2\Biggr) \\
 &= \dev(\tilde{I}^{(-)} ) \cdot \gamma(\tilde{I}^{(-)} ).
\end{align*}
Since $\epsilon \in (0,\nicefrac{1}{2}]$ and $\nicefrac{1}{16}+\nicefrac{1}{\sqrt{2}}\le \nicefrac{5}{6}$, \cref{le:bound_bad_indices} hence yields
\begin{align*}
    \tilde{\gamma}(\tilde{I}^{(-)} ) & \le \dev(\tilde{I}^{(-)} ) + \gamma(\tilde{I}^{(-)} ) + 2 \left(\dev(\tilde{I}^{(-)} ) \cdot \gamma(\tilde{I}^{(-)} )\right)^{\nicefrac{1}{2}}\\
    &\le \gamma(\tilde{I}^{(-)} ) + \frac{\epsilon^2}{8} \OPT_k(X) + 2 \left(\frac{\epsilon^2}{8} \OPT_k(X) \cdot \gamma(\tilde{I}^{(-)} )\right)^{\nicefrac{1}{2}}\\
    &\le \gamma(\tilde{I}^{(-)} ) + \left(\frac{\epsilon^2}{8} + \frac{\epsilon}{\sqrt{2}}\right) \cost(X,C,S)
    \le \gamma(\tilde{I}^{(-)} ) + \frac{5\epsilon}{6}  \cost(X,C,S),
\end{align*}
as claimed.
\end{proof}

We can now combine the results of this subsection to show property \labelcref{def:property_bprime}.

\begin{lemma} \label{le:coreset_property_b_line_resolution}
Let $\Delta = 2^{\tau^*} \error(\tau^*)$. Then
\begin{equation*}
        \cost\bigl(\coresetX,\coresetC,S\bigr) + \Delta \le (1+\epsilon)  \cost\bigl(X,C,S\bigr)
\end{equation*}
\end{lemma}

\begin{proof}
First, note that $2^{\tau^*}=|\tilde{I}^{(+)} |+|\tilde{I}^{(-)}|$. Using, in this order, \cref{rem:good_indices_equality,rem:total-dev}, and \cref{le:inequality_bad_indices,le:bound_bad_indices} we, hence, obtain
\begin{align*}
 \cost(\coresetX,\coresetC,S) & + \Delta
 = \tilde{\gamma}(\tilde{I}^{(+)}) +  |\tilde{I}^{(+)} | \cdot \error(\tau^*) + \tilde{\gamma}(\tilde{I}^{(-)}) + |\tilde{I}^{(-)} | \cdot \error(\tau^*)\\
& = \gamma(\tilde{I}^{(+)})+ \tilde{\gamma}(\tilde{I}^{(-)}) + \dev(\tilde{I}^{(-)}) \\
& \le  \gamma(\tilde{I}^{(+)})+ \gamma(\tilde{I}^{(-)})+ \frac{5\epsilon}{6}  \cost(X,C,S) + \frac{\epsilon}{6}  \cost(X,C,S)\\
&= (1+\epsilon) \cost(X,C,S),
\end{align*}
as claimed.
\end{proof}

Note that $\Delta$ depends only on the resolutions $\rho$ and $\tau$ but is independent of $k$ and $\Kappa$.

\subsubsection*{Resolution coresets in $[0,1]$}

The main result of this sections follows now easily.

\begin{theorem}\label{thm:one_dimensional_main_theorem}
$\bigl(X(\tau^*),\Omega (\tau^*)\bigr)$ is an $\epsilon$-coreset for $\bigl(k,X(\rho),\Omega(\rho),\Kappa\bigr)$ of size
	\begin{equation*}
	    |X(\tau^*)| \le 2^{\nicefrac{8}{3}}\frac{k}{\epsilon^{\nicefrac{2}{3}}}.
	\end{equation*}
\end{theorem}

\begin{proof}
By \cref{le:coreset_property_a_line_resolution,le:coreset_property_b_line_resolution}, $\bigl(X(\tau^*),\Omega (\tau^*)\bigr)$ is, indeed, an $\epsilon$-coreset and, by the choice of $\tau^*$, we have
	\begin{equation*}
	    |X(\tau^*)| = 2^{\tau^*} = 2^{\left\lceil \log\left(\frac{2k}{\epsilon^{\nicefrac{2}{3}}}\right) \right\rceil} \le 2\cdot \frac{2^{\nicefrac{5}{3}}k}{\epsilon^{\nicefrac{2}{3}}}= \frac{2^{\nicefrac{8}{3}}k}{\epsilon^{\nicefrac{2}{3}}}.
	\end{equation*}
This completes the proof of the theorem.
\end{proof}

\section{Proof of \cref{th:resolution-coreset}: The general case}\label{sec:proof-d}

We will now prove our main theorem by reducing it to the $1$-dimensional case settled already in \cref{sec:proof1D}. So suppose that $d\ge 2$. The key observation is that, in the isotropic case, the costs of a given clustering $C=(\xi_{ij})_{i\in [k],j\in [n]}\in \Cs_\Kappa(k, X,\Omega)$ on $X=X(\bm \rho)$ in $\R^d$ is a separable function with respect to the lines parallel to the coordinate axes. 

In fact, for $t\in [d]$, let $u_t$ denote the $t$th standard unit vector, and set
\begin{equation*}
   X^{(-t)}(\bm \rho) = X(\rho_1) \times \dots \times X(\rho_{t-1}) \times \{0\} \times X(\rho_{t+1}) \times \dots \times X(\rho_d).
\end{equation*}
Note that
\begin{equation*}
   X(\bm \rho) = \left(\frac{-1}{2^{\rho_t+1}}+\bigl[2^{\rho_t}\bigr] \cdot u_t\right)+ X^{(-t)}(\bm \rho),
\end{equation*}
i.e., $X^{(-t)}(\bm \rho)$ is the corresponding $(d-1)$-dimensional resolution set located in the coordinate hyperplane perpendicular to $u_t$.
Further,
\begin{equation*}
  \mathcal{L}^{(t)}=X^{(-t)}(\bm \rho) + \R\cdot u_t,
\end{equation*}
denote the set of lines through the Cartesian point set and parallel to the $t$th coordinate axis, and set $\mathcal{L}=\bigcup_{t\in [d]} \mathcal{L}^{(t)}$. Then we have
\begin{align*}
	\cost(X,C,\apdsites) 
	& = \sum_{t=1}^{d}\left( \sum_{i=1}^{k}\sum_{j=1}^{n} \xi_{ij} \omega_j \Bigl((x_{j})_{t}-(s_{i})_{t}\Bigr)^{2}\right)\\
	&= \sum_{t=1}^{d}\sum_{i=1}^{k}\sum_{L \in \mathcal{L}^{(t)}}\sum_{x_j\in L} \xi_{ij} \omega_j \Bigl((x_{j})_{t}-(s_{i})_{t}\Bigr)^{2}\\
	&= \sum_{t=1}^{d}\sum_{L \in \mathcal{L}^{(t)}} \left(\sum_{i=1}^{k}\sum_{x_j\in L} \nu(\bm \rho) \xi_{ij} \Bigl((x_{j})_{t}-(s_{i})_{t}\Bigr)^{2}\right).
\end{align*}

\subsubsection*{Clusterings induced on the lines in $\mathcal{L}$}

The double sum in parenthesis can be interpreted as the cost of a certain clustering of the points of $X(\bm \rho)\cap L$ on the line $L$. The reduction to the $1$-dimensional case is, however, somewhat technical and requires some additional notation. 

Before we introduce the details let us begin with some remarks concerning our notational conventions. Obviously, we need to distinguish between the $1$- and $d$-dimensional objects and also bear in mind, that in the $1$-dimensional setting points carry different weights than in the $d$-dimensional situation. In addition, we need to consider sets in the $(d-1)$-dimensional coordinate hyperplanes. In particular, batches of points in $X^{(-t)}(\bm \rho)$ are relevant. In order to keep the notation as intelligible as possible, we will use subscripts for components (as in ${\bm \rho}=(\rho_1,\ldots,\rho_{d})$) but also for \enquote{intrinsically $1$-dimensional} objects. For instance, we write $X_L$ for the point set $X(\bm \rho)\cap L$. Naturally, $X_L$ can be identified with the resolution set $X(\rho_t)$ in $\R$. Note, however, that each point in $X(\rho_t)$ carries the weight $\nu(\rho_t)$ (rather than $\nu(\bm \rho)$ as it is the case for subsets of $X(\bm \rho)$). Further, we use superscripts (in parenthesis) to indicate the dependence of sets in $\R^d$ on a parameter, as e.g. in $\mathcal{L}^{(t)}$. We indicate that a single component is missing by means of a minus sign, and write, for instances, ${\bm \rho}^{(-t)}=(\rho_1,\ldots,\rho_{t-1},\rho_{t+1}\ldots,\rho_{d})$. The same notation is also used if a component is fixed. In $X^{(-t)}(\bm \rho)$, for example, the minus sign indicates that the $t$th coordinate of each point in $X(\bm \rho)$ is set to $0$. In other situations, again indicated by the superscript $^{(-t)}$, the $t$th coordinate may, however, be fixed differently. While these conventions are used to facilitate a more intuitive exposition, all objects will, of course, be defined precisely. 

Let us now introduce the needed concepts more precisely. We begin with the formal definition of clusterings on the lines $L\in \mathcal{L}^{(t)}$ \emph{induced} by $C$. Let 
\begin{equation*}
    N_L= \bigl\{j\in [n]: x_j\in L\bigr\}, \qquad I_{C,L}= \bigl\{i\in [k]: \exists j\in N_L: \xi_{ij}\ne 0 \bigr\}, \qquad k_L=\bigl|I_{C,L}\bigr|,
\end{equation*}
and set
\begin{equation*}
\begin{array}{cllrl}
     (\kappa_i)_L&= \sum_{j\in N_L} \nu(\rho_t) \xi_{ij} \quad & \text{for $i\in I_{C,L}$}, \qquad &
     \Kappa_L &= \bigl\{(\kappa_i)_L: i\in I_{C,L}\bigr\},\\[.1cm]
     (\omega_j)_L&= \nu(\rho_t) \quad & \text{for $j\in N_L$}, \qquad & \Omega_L &= \bigl\{(\omega_j)_L: j\in N_L\bigr\}.
\end{array}
\end{equation*}
Then, of course, 
\begin{equation*}
    C_L= \bigl((C_i)_L : i\in I_{C,L}\bigr)= \bigl( \xi_{ij} : i\in I_{C,L}, j\in N_L\bigr) \in \Cs_{\Kappa_L}(k_L, X_L,\Omega_L).
\end{equation*}
Note that 
\begin{equation*}
    \sum_{i\in I_{C,L}} (\kappa_i)_{L} = \sum_{i\in I_{C,L}} \sum_{j\in N_L} \xi_{ij} (\omega_j)_L = \sum_{x_j\in X_L} \nu(\rho_t) = \nu(\rho_t) \cdot |X_L| = 1.
\end{equation*}
Further, let $(s_i)_L$ denote the orthogonal projection of $s_i$ on $L$, and set
\begin{equation*}
    S_L=\bigl\{(s_i)_L: i\in I_{C,L}\bigr\}.
\end{equation*} 
Let us point out that $(s_i)_L=(s_\ell)_L$ whenever $(s_{i})_{t}=(s_{\ell})_{t}$, i.e. $S_L$ may, in general, be a multiset. In such a situation it is straightforward to merge the corresponding cells of $C_L$, perform subsequent operations in the standard model and, finally, split-up the assignment again. We prefer, however, not to further complicate the notation and simply apply this additional technicality tacitly. Note that, when $C$ is integer and admits a strongly compatible power diagram, the sites $(s_i)_L$ for $i\in I_{C,L}$ are all different anyway.

With the introduced notation we see that
\begin{align*}
\sum_{i=1}^{k}&\sum_{x_j\in L} \nu(\bm \rho) \xi_{ij} \Bigl((x_{j})_{t}-(s_{i})_{t}\Bigr)^{2}
=\sum_{i\in S_{C,L}} \sum_{j\in N_L} \nu(\bm \rho) \xi_{ij} \bigl\|x_{j}-(s_{i})_{L}\bigr\|^{2} \\
& =\nu({\bm \rho}^{(-t)}) \sum_{i\in S_{C,L}} \sum_{j\in N_L} \nu(\rho_t) \xi_{ij} \bigl\|x_{j}-(s_{i})_{L}\bigr\|^{2} = \nu({\bm \rho}^{(-t)}) \cost(X_L,C_L,S_L) .
\end{align*}
and hence have the following result.

\begin{remark} \label{rem:decoupling_cost_into_lines}
    \begin{equation*}
    \cost(X,C,\apdsites) =  \sum_{t=1}^{d}\sum_{L \in \mathcal{L}^{(t)}} \nu({\bm \rho}^{(-t)}) \cost\bigl(X_L,C_L,S_L\bigr).
   \end{equation*}
\end{remark}

\cref{rem:decoupling_cost_into_lines} gives the desired split of the costs of $C$ into the costs of the induced clusterings $C_L$ for $L \in \mathcal{L}$ and opens up the possibility to apply results from \cref{sec:proof1D}.

\subsubsection*{Resolution coresets}

As before we will use the simplified notation $X=X(\bm \rho)$, $\coresetX=X(\bm \tau)$, and generally signify objects related to $\coresetX$ by means of $\tilde{\quad}$ put on top. In particular, we set
\begin{equation*}
    \tilde{\mathcal{L}}^{(t)}=\coresetX^{(-t)}+\R\cdot u_t \quad \text{ for $t\in [d]$}, \qquad \tilde{\mathcal{L}}=\bigcup_{t\in [d]} \tilde{\mathcal{L}}^{(t)},
\end{equation*}
and let $\mathcal{B}^{(-t)}$ denote the set of batches in $X^{(-t)}$ induced by $\tilde{X}^{(-t)}$ by. Further, for $B\in \mathcal{B}^{(-t)}$, let $\mathcal{L}^{(t)}(B)= B+\R\cdot u_t$. For future reference, the following remark collects the cardinalities of the different sets.

\begin{remark} \label{rem:cardinalities}
We have $\nu\bigl({\bm\tau}^{(-t)}\bigr)\cdot\bigl|\mathcal{B}^{(-t)} \bigr| = 1$
and, for $B\in \mathcal{B}^{(-t)}$,
\begin{equation*}
 \qquad |B| = \frac{\nu\bigl({\bm\tau}^{(-t)}\bigr)}{\nu\bigl({\bm\rho}^{(-t)}\bigr)} = \prod_{\ell\in [d]\setminus \{t\}} 2^{\rho_\ell-\tau_\ell} = \bigl|\mathcal{L}^{(t)}(B) \bigr|, \qquad
 \bigl|\mathcal{L}^{(t)} \bigr|  =  \bigl|\mathcal{L}^{(t)}(B)\bigr| \cdot \bigl|\tilde{\mathcal{L}}^{(t)} \bigr|
\end{equation*}
\end{remark}

In fact, the lines of $\tilde{\mathcal{L}}^{(t)}$ are in $1$-to-$1$ correspondence with the batches in $\mathcal{B}^{(-t)}$ as each line $\tilde{L} \in \tilde{\mathcal{L}}^{(t)}$ intersects exactly one such batch. In the following, we use the notation $\tilde{L}=\tilde{L}(B)$ and $B=B(\tilde{L})$ to indicate this correspondence. Note that the point $\tilde{L}\cap X^{(-t)}\in B$ is the centroid of $B$. Also, recall from \cref{sec:resolution_coreset_prelimiaries} that the merging function $p:X\rightarrow \coresetX$ has a Cartesian structure, too. More precisely, for $x\in X$, we have
\begin{equation*}
    p(x)= \Bigl(p_1\bigl((x)_1\bigr), \ldots, p_d\bigl((x)_d\bigr)\Bigr).
\end{equation*}
Hence merging takes place in each coordinate independently. 

In the following, we will use the specific values
\begin{equation*}
\tau^*_1=\ldots =\tau^*_d=\tau^* = \left\lceil \log\left(\frac{2^{\nicefrac{5}{3}}k^*}{\epsilon^{\nicefrac{2}{3}}}\right) \right\rceil, \qquad
\bm \tau = \bm \tau^* = (\tau^*_1,\ldots,\tau^*_d), \qquad \coresetX=X(\bm \tau^*),
\end{equation*}
and define
\begin{equation*}
\Delta_t =2^{\tau_t^*} \error(\tau_t^*) \quad (t\in [d]), \qquad
    \Delta = \sum_{t=1}^{d}\sum_{L \in \mathcal{L}^{(t)}} \nu(\bm \rho^{(-t)}) \Delta_t.
\end{equation*}
Finally we set  $k^*=k$ and note that $k_L\le k^*$ for each $L \in \mathcal{L}$. Since 
\begin{equation*}
        |\coresetX| = |X(\bm \tau^*)|^d  \le \left(\frac{2^{\nicefrac{8}{3}}k^*}{\epsilon^{\nicefrac{2}{3}}}\right)^d,
    \end{equation*}
\cref{th:resolution-coreset} is proved once the two coreset properties
are established. This will be done in the next two subsections, more specifically in \cref{le:coreset_property_a_d_dimensions,le:coreset_property_b_d_dimensions}. To (slightly) simplify the notation we will omit the asterisk and simply write $\bm \tau=(\tau_1,\ldots,\tau_d)$ for the resolution of $\coresetX$.

\subsubsection*{Coreset property \cref{def:property_a}}

Let $\coresetC=(\tilde{\xi}_{iq})_{i\in [k],q\in [\tilde{n}]}\in \Cs_\Kappa(k, \coresetX,\coresetOmega)$ be a clustering on $\coresetX$. Recall from \cref{sec:resolution_coreset_prelimiaries} that $\coresetC$ can be extended to a clustering $C=g(\coresetC)=(\xi_{ij})_{i\in [k],q\in [n]}\in \Cs_\Kappa(k, X,\Omega)$ by means of the extension $\extension: \Cs_\Kappa(k,\coresetX,\coresetOmega) \rightarrow \Cs_\Kappa(k,X,\Omega)$
defined by 
\begin{equation*}
\xi_{ij} = \tilde{\xi}_{iq} \qquad \text{for $j \in [n]$, $i\in [k]$, $q \in [\tilde{n}]$ with $p(x_j)=\tilde{x}_q$}.
\end{equation*}

In order to apply the established $1$-dimensional results we will now take a closer look at the relation between extensions in $\R^d$ and on the lines. Of course,  \cref{rem:decoupling_cost_into_lines} can be used to express $\cost(\coresetX,\coresetC,\apdsites)$ in terms of the costs of the induced clusterings on the lines $\tilde{L}\in \tilde{\mathcal{L}}$. Explicitly, it reads as follows.

\begin{remark} \label{rem:decoupling-coreset}
    \begin{equation*}
    \cost(\coresetX,\coresetC,\apdsites) =  \sum_{t=1}^{d}\sum_{\tilde{L} \in \tilde{\mathcal{L}}^{(t)}} \nu({\bm \tau}^{(-t)}) \cost\bigl(\coresetX_{\tilde{L}},\coresetC_{\tilde{L}},S_L\bigr).
   \end{equation*}
\end{remark}
Hence, using the coreset properties on the involved lines alone we can bound $\cost\bigl(X_L,C_L,S_L\bigr)$ only for $L\in \tilde{\mathcal{L}}$. Indeed, the underlying point set is $X\cap \bigcup_{\tilde{L}\in \tilde{\mathcal{L}}} {\tilde L}$ but does not contain any point from $X\setminus \bigcup_{\tilde{L}\in \tilde{\mathcal{L}}} \tilde{L}$. 
Due to the definition of the extension, $\cost\bigl(X_L,C_L,S_L\bigr)$ can, however, be included for all lines in $\mathcal{L}\setminus \tilde{\mathcal{L}}$ as well. In fact, except of a translation of the point set, the clusterings $C_{L_1}$ and $C_{L_2}$ for lines $L_1,L_2 \in \mathcal{L}^{(t)}$ coincide whenever there is a batch $B\in \mathcal{B}^{(-t)}$ such that $L_1,L_2 \in \mathcal{L}^{(t)}(B)$. Hence, in particular, we have the following remark.

\begin{remark} \label{rem:reverse_clustering_is_identical}
Let $\tilde{L}\in \tilde{\mathcal{L}}^{(t)}$, $B=B(\tilde{L})\in {\mathcal{B}}^{(t)}$, and $L\in {\mathcal{L}}^{(t)}(B)$. Then 
\begin{equation*}
        \cost (X_{L},C_{L},S_{L}) = 
        \cost (X_{\tilde{L}},C_{\tilde{L}},S_{\tilde{L}}) 
\end{equation*}
\end{remark}

The following lemma addresses coreset property \cref{def:property_a}. 

\begin{lemma}\label{le:coreset_property_a_d_dimensions}
Let $\coresetC\in \Cs_\Kappa(k, \coresetX,\coresetOmega)$ and $C=g(\coresetC)$. Then
\begin{equation*}
        \cost\bigl(X,C,S\bigr) = \cost\bigl(\coresetX,\coresetC,S\bigr) + \Delta.
\end{equation*}
\end{lemma}

\begin{proof}
By \cref{rem:cardinalities} we have 
$\nu\bigl({\bm\rho}^{(-t)}) \cdot |\mathcal{L}^{(t)}| = \nu\bigl({\bm\tau}^{(-t)}\bigr) \cdot |\tilde{\mathcal{L}}^{(t)}|$, for each $t\in [d]$. Hence,
\begin{equation*}
    \Delta = \sum_{t=1}^{d} \sum_{L \in \mathcal{L}^{(t)}} \nu(\bm \rho^{(-t)}) \Delta_t
    = \sum_{t=1}^{d} \sum_{\tilde{L} \in \tilde{\mathcal{L}}^{(t)}} \nu({\bm \tau}^{(-t)}) \Delta_t.
\end{equation*}
Thus with \cref{rem:decoupling-coreset,rem:cardinalities} we obtain
\begin{align*}
    \cost &\bigl(\coresetX,\coresetC,S\bigr) + \Delta = \sum_{t=1}^{d}\sum_{\tilde{L} \in \tilde{\mathcal{L}}^{(t)}} \nu(\bm \tau^{(-t)}) \left( \cost\bigl(\coresetX_{\tilde{L}},\coresetC_{\tilde{L}},S_{\tilde{L}}\bigr) + \Delta_t \right) \\
    &= \sum_{t=1}^{d}\sum_{\tilde{L} \in \tilde{\mathcal{L}}^{(t)}} \frac{1}{|B(\tilde{L})|} \sum_{L \in {\mathcal L}^{(t)}(B(\tilde{L}))} \nu(\bm \tau^{-t}) \left( \cost\bigl(\coresetX_{\tilde{L}},\coresetC_{\tilde{L}},S_{\tilde{L}}\bigr)+ \Delta_t\right)\\
     &= \sum_{t=1}^{d}\sum_{\tilde{L} \in \tilde{\mathcal{L}}^{(t)}} \sum_{L \in {\mathcal L}^{(t)}(B(\tilde{L}))} \nu(\bm \rho^{-t}) \left( \cost\bigl(\coresetX_{\tilde{L}},\coresetC_{\tilde{L}},S_{\tilde{L}}\bigr)+ \Delta_t\right)
\end{align*}
Now, with ${\tilde{L}}\in \tilde{\mathcal{L}}^{(t)}$, $B=B(\tilde{L})\in {\mathcal{B}}^{(t)}$, and $L\in {\mathcal{L}}^{(t)}(B)$, we apply \cref{le:coreset_property_a_line_resolution} to the clustering $\coresetC_{\tilde{L}}\in\Cs_{\Kappa_{\tilde{L}}}(k_{\tilde{L}},\coresetX_{\tilde{L}},\coresetOmega_{\tilde{L}})$. With the aid of \cref{rem:reverse_clustering_is_identical}, obtain
\begin{equation*}
        \cost (X_L,C_L,S_L)=
        \cost (X_{\tilde{L}},C_{\tilde{L}},S_{\tilde{L}}) = 
        \cost (\coresetX_{\tilde{L}},\coresetC_{\tilde{L}}, S_{\tilde{L}}) + \Delta_t.
\end{equation*}
Hence, with \cref{rem:decoupling_cost_into_lines}, we can now reverse the previous arguments to see that
\begin{align*}
\sum_{t=1}^{d} & \sum_{\tilde{L} \in \tilde{\mathcal{L}}^{(t)}(\bm \tau)} \sum_{L \in {\mathcal L}^{(t)}(B(\tilde{L}))} \nu(\bm \rho^{-t}) \left( \cost\bigl(\coresetX_{\tilde{L}},\coresetC_{\tilde{L}},S_{\tilde{L}}\bigr)+ \Delta_t\right)\\
&= \sum_{t=1}^{d}\sum_{\tilde{L} \in \tilde{\mathcal{L}}^{(t)}} \sum_{L \in {\mathcal L}^{(t)}(B(\tilde{L})} \nu(\bm \rho^{-t}) \cost (X_L,C_L,S_L)\\
&= \sum_{t=1}^{d}\sum_{L \in \mathcal{L}^{(t)}} \cost (X_L,C_L,S_L) = \cost (X,C,S).
\end{align*}
This completes the proof.
\end{proof}

\subsubsection*{Coreset property \cref{def:property_bprime}}

Let us now turn to coreset property \cref{def:property_bprime}.
So, suppose that the clustering $C\in \Cs_{\Kappa}(k,X,\Omega)$ is integer and admits a strongly compatible power diagram. Then, for each line $L\in \mathcal{L}$ each clustering
\begin{equation*}
    C_L= \bigl( \xi_{ij} : i\in I_{C,L}, j\in N_L\bigr) \in \Cs_{\Kappa_L}(k_L, X_L,\Omega_L)
\end{equation*}
is also integer and admits a strongly compatible power diagram. We can hence apply \cref{le:coreset_property_b_line_resolution} to relate its cost to that of the corresponding coreset clustering. Note that the latter lives on $L$ while $\coresetC_{\tilde{L}}$ is only defined on the lines $\tilde{L}\in \tilde{\mathcal{L}}$. The following proof of coreset property \cref{def:property_bprime} will therefore explicitly address the transition from coresets on each line $L\in \mathcal{L}$ to $\coresetC$.

\begin{lemma}\label{le:coreset_property_b_d_dimensions}
Let $C\in \Cs_{\Kappa}(k,X,\Omega)$ be integer and admit a strongly compatible power diagram, and set $\coresetC=p(C)$. Then
\begin{equation*}
        \cost\bigl(\coresetX,\coresetC,S\bigr) + \Delta \le (1+\epsilon)  \cost\bigl(X,C,S\bigr).
\end{equation*}
\end{lemma}

\begin{proof} By \cref{rem:decoupling_cost_into_lines}, $\cost(X,C,\apdsites)$ can be decomposed into the costs on the lines $L\in \mathcal{L}$, i.e., $\cost\bigl(X_L,C_L,S_L\bigr)$. For each $C_L\in \Cs_{\Kappa_L}(k_L,X_L,\Omega_L)$ we derive a coreset clustering with $k_l\le k^*$ according to \cref{le:coreset_property_b_line_resolution}. We refrain from denoting it by $\widetilde{C_L}$ (with a wide tilde) but use a different notation for better distinction. Let $t\in [d]$, $L\in \mathcal{L}^{(t)}$ and $x_L^{(-t)}\in X^{(-t)}$ such that $L=x_L^{(-t)}+\R\cdot u_t$. Then we set
\begin{equation*}
\begin{array}{cl}
Y_L& = x_L^{(-t)} + \bigr\{x_t u_t: x_t\in X(\tau_t)\bigr\}, \qquad  \coresetOmega_L=\bigl(\nu(\tau_t),\ldots, \nu(\tau_t)\bigr)\\[.1cm]
Z_L& =(\zeta^{(L)}_{ir})_{i\in I_{C,L}, r \in [2^{\tau_t}]} \in \Cs_{\Kappa_{L}}(k_L,Y_L,\coresetOmega_{L}).
\end{array}
\end{equation*}
Note that 
\begin{equation*}
    Y_L=x_L^{(-t)}+\{y_1,\ldots,y_{2^{\tau_t}}\}\cdot u_t \quad \text{with}
    \quad y_r=2^{-(\tau_t+1)}+ (r-1)2^{-\tau_t} \quad (r\in \bigl[2^{\tau_t}]\bigr).
\end{equation*}
By \cref{rem:decoupling_cost_into_lines} and \cref{le:coreset_property_b_line_resolution} we have
\begin{align*}
     (1+\epsilon) \cost(X,C,\apdsites)  - \Delta & = \sum_{t=1}^{d}\sum_{L \in \mathcal{L}^{(t)}} \nu(\bm \rho^{(-t)}) \left( (1+\epsilon) \cost\bigl(X_L,C_L,S_L\bigr) - \Delta_t \right) \\
     & \ge \sum_{t=1}^{d}\sum_{L \in \mathcal{L}^{(t)}} \nu(\bm \rho^{(-t)}) \cost\bigl(Y_L,Z_L,S_L\bigr) \\
     & = \sum_{t=1}^{d}\sum_{B \in \mathcal{B}^{(-t)}} \sum_{L \in \mathcal{L}^{(t)}(B)} \nu(\bm \rho^{(-t)}) \cost\bigl(Y_L,Z_L,S_L\bigr). 
\end{align*}
Now, let $q\in [\tilde{n}]$, $B=(B_q+ \R\cdot u_t)\cap X^{(-t)}\in \mathcal{B}^{(t)}$, $\tilde{L}=\tilde{L}(B)$, and $r=r(q)\in [2^{\tau_t}]$ such that $(\tilde{x}_q)_t=(y_r)_t$. Then, as $\tilde{x}_q\in \tilde{L}$, we have
\begin{equation*}
    \norm{y_r - (s_i)_L}_2^2 = \norm{\tilde{x}_q - (s_i)_{\tilde{L}}}_2^2 \quad \text{for all $L\in \mathcal{L}^{(t)}(B)$}
\end{equation*}
and 
\begin{align*}
    \tilde{\xi}_{iq} &= \frac{1}{\nu(\bm \tau)} \sum_{x_j\in B_q} \xi_{ij}\nu(\bm \rho) 
    = \frac{\nu(\bm \rho)}{\nu(\bm \tau)} \sum_{L \in \mathcal{L}^{(t)}(B)} \sum_{x_j\in L \cap B_q}\xi_{ij}\\[.1cm]
     & =  \frac{\nu(\bm \rho)}{\nu(\bm \tau)} \sum_{L \in \mathcal{L}^{(t)}(B)} \frac{\nu(\tau_t)}{\nu(\rho_t)} \left(\frac{1}{\nu(\tau_t)} \sum_{x_j\in L \cap B_q}\xi_{ij}\nu(\rho_t)\right)
      = \frac{\nu(\bm \rho^{(-t})}{\nu(\bm \tau^{(-t)})} \sum_{L \in \mathcal{L}^{(t)}(B)} \zeta_{ir}.
\end{align*}
Hence, with the additional setting $\zeta^{(L)}_{ir}=0$ for $i\not\in I_{C,L}$, we obtain 
\begin{align*}
     \sum_{L \in \mathcal{L}^{(t)}(B)} & \nu(\bm \rho^{(-t)}) \cost\bigl(Y_L,Z_L,S_L\bigr)
     =  \sum_{L \in \mathcal{L}^{(t)}(B)} \nu(\bm \rho^{(-t)}) \sum_{i\in I_{C,L}} \sum_{r=1}^{2^{\tau_{t}}} \zeta^{(L)}_{ir} \nu(\tau_t) \norm{y_r - (s_i)_{L}}_2^2 \\
     &= \sum_{i=1}^{k} \sum_{\tilde{x}_q \in \coresetX_{\tilde{L}}} \left( \nu(\bm \rho^{(-t)})  \sum_{L \in \mathcal{L}^{(t)}(B)} \zeta^{(L)}_{ir(q)} \right) \nu(\tau_t) \norm{\tilde{x}_q - (s_i)_{\tilde{L}}}_2^2 \\
     &= \nu(\bm \tau^{(-t)}) \sum_{i=1}^k \sum_{\tilde{x}_q \in \coresetX_{\tilde{L}}}  \tilde{\xi}_{iq} \nu(\tau_t) \norm{\tilde{x}_q - (s_i)_{\tilde{L}}}_2^2 = \nu(\bm \tau^{-t}) \cost(\coresetX_{\tilde{L}},\coresetC_{\tilde{L}},S_{\tilde{L}}).
\end{align*}
Thus,
\begin{align*}
     (1+\epsilon) \cost&(X,C,\apdsites)  - \Delta \ge \sum_{t=1}^{d}\sum_{B \in \mathcal{B}^{(-t)}} \sum_{L \in \mathcal{L}(B)} \nu(\bm \rho^{(-t)}) \cost\bigl(Y_L,Z_L,S_L\bigr)\\
     & = \sum_{t=1}^{d}\sum_{\tilde{L}\in \mathcal{L}^{(-t)}} \nu(\bm \tau^{(-t)}) \cost(\coresetX_{\tilde{L}},\coresetC_{\tilde{L}},S_{\tilde{L}}) = \cost\bigl(\coresetX,\coresetC,S\bigr),
\end{align*}
which concludes the proof.
\end{proof}

\section{Proof of \cref{th:approx-resolution}}\label{sec:approx-resolution-proof}
While \cref{th:resolution-coreset} explicitly exploits the orthogonality of the Euclidean norm and hence produces resolution coresets only for the Euclidean case we will now show that it can still be used for general anisotropic objective functions defined by means of families $\apdmatrices=\{A_1,\ldots,A_k\}$ of positive definite symmetry $(d\times d)$-matrices. The next well-known result specifies how good constrained clusterings on coresets yield good clusterings on the original data sets. In the proof of \cref{th:approx-resolution} we will need it only in the Euclidean case and for $\delta=1$. We formulate it here for $(\epsilon,\delta)$-coresets in the anisotropic case since it provides the rationale behind the general coreset definition given in \cref{sec:preliminaries_resolution_coreset} which involves two $\Delta$-terms.

\begin{proposition}[ {\cite[Thm. 3.5 (a)]{FG22}} ] 
\label{prop:approximation_using_coreset}
    Let $I=(k,X,\Omega,\apdmatrices,\Kappa,S)$ be an instance of \wcaa, $\epsilon \in (0,\nicefrac{1}{2}]$, $\delta \in [1,\infty)$, and $\gamma \ge \delta$. Further, let $(\coresetX,\coresetOmega)$ be an $(\nicefrac{\epsilon}{3},\delta)$-coreset for $I_S$, and $\extension$ its extension. 
    
    Now suppose that $\coresetC \in \Cs_\Kappa(k,\coresetX,\coresetOmega)$ is a $\gamma$-approximation for $\coresetinstance=(k,\coresetX,\coresetOmega,\apdmatrices,\Kappa,S)$, then $\extension(\coresetC)$ is a $(1+\epsilon)\gamma $-approximation for $I$.
\end{proposition}

We can now prove \cref{th:approx-resolution}.
\medskip

{\normalfont\scshape Proof of \cref{th:approx-resolution}:}
First note that for a positive definite symmetric $(d\times d)$-matrix $A$ with smallest and largest eigenvalues $\lambda^{-}(A)$ and $\lambda^{+}(A)$, respectively, we have  
$$
\lambda^{-}(A)\cdot \norm{x}_2^{2} \le \norm{x}_{A}^{2} \le \lambda^{+}(A)\cdot\norm{x}_2^{2} \qquad (x\in \R^d).
$$
Recalling that  
\begin{equation*}
    \lambdamin = \min \bigl\{\lambda^{-}(A_i): i\in [k]\bigr\}, \qquad  \lambdamax = \max \bigl\{\lambda^{+}(A_i): i\in [k]\bigr\}.
\end{equation*}
we hence have
\begin{align*}
\cost_\AK(X,C,\apdsites) &= \displaystyle \sum_{i=1}^{k}\sum_{j=1}^{n} \xi_{ij} \omega_j \norm{x_{j}-s_{i}}_{\apdmatrix_i}^{2} \le 
\sum_{i=1}^{k}\sum_{j=1}^{n} \xi_{ij} \omega_j \lambdamax \norm{x_{j}-s_{i}}_{2}^{2}\\
& = \lambdamax \cdot \cost(X,C,\apdsites),
\end{align*}
and similarly,
\begin{equation*}
 \lambdamin \cdot \cost(X,C,\apdsites) \le \cost_\AK(X,C,\apdsites).
\end{equation*}
Since $C=\extension(\coresetC)$, \cref{prop:approximation_using_coreset} therefore yields
    \begin{align*}
\cost_\AK(X,C,S)& \le \lambdamax \cdot \cost(X,C,S) \le \lambdamax (1+\epsilon)\gamma \cdot \cost(X,S)\\
&\le (1+\epsilon)\gamma \frac{\lambdamax }{\lambdamin} \cdot \cost_\AK(X,S),
    \end{align*}
which completes the proof.

\section{Final remarks}\label{sec:conclusion}

As pointed out before, the size bounds for resolution coresets are superior to those for pencil coresets in small dimensions. Comparing the results of \cref{th:resolution-coreset} and \cref{prop:smaller_coreset} we see, however, that the number of clusters enters as $k^d$ and $k^2$, respectively. In fact, resolution coresets utilize the grid structure of $X$ but are themselves restricted to a coarser grid. Pencil coresets, on the other hand, produce very different point densities depending on the distance from the chosen vertices of the pencils. Since, for resolution coresets, we insist on a \emph{uniform} resolution on each line, we cannot employ the rationale leading to the reduced dependence on $k$ in \cref{prop:smaller_coreset}. It remains, however, an open question that is already relevant in dimension $3$ whether the term $k^d$ can be further reduced. Also, generalizations to Cartesian point sets with arbitrary positive weights would be interesting. Further, an analysis of nonuniform tilings of $[0,1]^d$ into axis parallel boxes would be relevant. This corresponds to adaptive thinning of Cartesian point sets and is interesting for fast image processing; see e.g. \cite{AF21} and the papers quoted there.

	\printbibliography

\end{document}